\documentclass[journal]{IEEEtran}
\usepackage{cite}

\usepackage{amsmath,amssymb,amsfonts,amsthm,stmaryrd,xcolor}
\usepackage{graphicx}
\usepackage{algorithm,algorithmic}

\usepackage{mathtools}

\mathtoolsset{showonlyrefs}

\def\N{{\mathbb N}}    
\def\Z{{\mathbb Z}}     
\def\R{{\mathbb R}}    
\def\C{{\mathbb C}}    

\def\cA{{\mathcal A}} 
\def\cB{{\mathcal B}}

\def\cL{{\mathcal L}}
\def\cM{{\mathcal M}} 
 
\def\cO{{\mathcal O}}

\newcommand{\diag}{\operatorname{diag}}
\newcommand{\im}{\operatorname{Im}}

\newcommand{\triplenorm}[1]{\left\vert\kern-0.25ex\left\vert\kern-0.25ex\left\vert #1 
	\right\vert\kern-0.25ex\right\vert\kern-0.25ex\right\vert}

\def\idty{{\operatorname{Id}}}

\theoremstyle{plain}
\newtheorem{theorem}{Theorem}[section] 
\newtheorem{proposition}[theorem]{Proposition}
\newtheorem{lemma}[theorem]{Lemma}
\newtheorem{corollary}[theorem]{Corollary}

\theoremstyle{definition}
\newtheorem{definition}{Definition}
\newtheorem{notations}[definition]{Notation}
\newtheorem{remark}[theorem]{Remark}
\newtheorem{assumption}[theorem]{Assumptions}

\usepackage{hyperref}

\hypersetup{
	colorlinks=true,
	linkcolor=blue,    
	citecolor=blue,    
	urlcolor=blue      
}

\hypersetup{
	pdfauthor={Cyprien Tamekue, Ruiqi Chen, and ShiNung Ching},
	pdfsubject={Paper manuscrit},
	pdftitle={On the control of recurrent neural networks using constant inputs},
	pdfkeywords={Control Synthesis, Nonlinear Control, Recurrent Neural Networks, Neurostimulation, Transcranial Direct Current Stimulation},
	colorlinks=true,
	linkcolor=blue,    
	citecolor=blue,    
	urlcolor=blue      
}

\usepackage{textcomp}
\def\BibTeX{{\rm B\kern-.05em{\sc i\kern-.025em b}\kern-.08em
		T\kern-.1667em\lower.7ex\hbox{E}\kern-.125emX}}
\begin{document}
	\title{On the control of recurrent neural networks using constant inputs}
	
	\author{Cyprien Tamekue, Ruiqi Chen, and ShiNung Ching, \IEEEmembership{Senior Member, IEEE}
		\thanks{This work is partially supported by grant R21MH132240 from the US National Institutes of Health to SC.}
		\thanks{ Cyprien Tamekue and ShiNung Ching are with the Department of Electrical and Systems Engineering, Washington University in St. Louis, St. Louis, 63130, MO, USA; Ruiqi Chen is with the Neurosciences Program in the Division of Biology and Biomedical Sciences, Washington University in St. Louis, St. Louis, 63130, MO, USA (e-mail: cyprien@wustl.edu, chen.ruiqi@wustl.edu, shinung@wustl.edu). }
	}
	
	\maketitle
	\begin{abstract}
		This paper investigates the controllability of a broad class of recurrent neural networks widely used in theoretical neuroscience, including models of large-scale human brain dynamics. Motivated by emerging applications in non-invasive neurostimulation such as transcranial direct current stimulation (tDCS), we study the control synthesis of these networks using constant and piecewise constant inputs. The neural model considered is a continuous-time Hopfield-type system with nonlinear activation functions and arbitrary input matrices representing inter-regional brain interactions. Our main contribution is the formulation and solution of a control synthesis problem for such nonlinear systems using specific solution representations. {These representations yield explicit algebraic conditions for synthesizing constant and piecewise constant controls that solve a two-point boundary value problem in state space up to higher-order corrections with respect to the time horizon. In particular, the input is constructed to satisfy a tractable small-time algebraic relation involving the Jacobian of the nonlinear drift, ensuring that the synthesis reduces to verifying conditions on the system matrices. For canonical input matrices that directly actuate $k$ nodes, this implies that the reachable set (with constant inputs) of a given initial state is an affine subspace whose dimension equals the input rank and whose basis can be computed efficiently using a thin QR factorization.} Numerical simulations illustrate the theoretical results and demonstrate the effectiveness of the proposed synthesis in guiding the design of brain stimulation protocols for therapeutic and cognitive applications.
		
	\end{abstract}
	
	\begin{IEEEkeywords}
		Control Synthesis, Nonlinear Control, Recurrent Neural Networks, Neurostimulation, Transcranial Direct Current Stimulation.
	\end{IEEEkeywords}
	
	\section{Introduction}
	\label{sec:introduction}
	\IEEEPARstart{N}{on-invasive} neurostimulation techniques, such as transcranial magnetic stimulation (TMS) and transcranial electrical stimulation (tES), are increasingly being used to modulate brain activity in order to achieve desired cognitive outcomes~\cite{bergmann2009acute,lopez2014inter,grover2022long}. Conceptualizing the brain as a controlled dynamical system~\cite{medaglia2019clarifying,schiff2011neural} offers a powerful framework for identifying key brain regions and networks~\cite{power2013evidence} involved in specific cognitive functions and understanding how they can be modulated using targeted stimulation.
	
	In this regard, there is strong potential at the intersection of network control theory and clinical and basic neuroscience. For instance, there have been successes in the application of linear network control theory to optimize TMS, administered with a constant input over short time intervals~\cite{muldoon2016stimulation}. Such applications target the control of specific brain regions based on structural network connectivity derived from diffusion spectrum imaging (DTI). Similarly, at an analysis level, linear network control theory has been employed in human neuroimaging studies, again by leveraging DTI-parameterized models~\cite{gu2015controllability}. In this latter study, the authors postulated how specific brain regions may be suitable targets for control based on advantageous controllability properties relative to their potential actuation. A complementary perspective views brain stimulation---particularly deep brain stimulation---as a form of vibrational control that stabilizes desired neural dynamics through high-frequency inputs, a notion that has been formalized and validated in recent theoretical work~\cite{qin2022vibrational, qin2023vibrational}.
	
	While these earlier studies~\cite{gu2015controllability,muldoon2016stimulation} have provided valuable insights by using linear network control theory and structural connectivity data to understand brain dynamics, they are limited in key respects. Foremost, by virtue of linearity, they are only assured to provide local characterizations. In~\cite{muldoon2016stimulation}, it is shown that these local analyses can offer some predictive power over how stimulation induces functional changes in nonlinear models. However, the changes considered are in terms of low-dimensional correlation metrics between brain areas (network nodes), as opposed to specific configurations in state space. In contrast, our goal is to enable formal control synthesis to induce arbitrary network states in the presence of full nonlinearity, which is likely essential for understanding how different brain dynamics mediate cognitive function.
	
	\begin{figure}[t]
		\centering
		\includegraphics[width=0.80\linewidth]{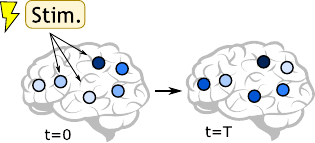}
		\caption{\textbf{We consider the control of network models of the general form \eqref{eq:Hopfield i-th neuron}. We are motivated by emerging applications in neurostimulation involving the manipulation of brain networks by means of exogenous input.}}
		\label{fig:scheme}
	\end{figure}
	
	Indeed, the need for more biologically realistic, nonlinear models of large-scale brain dynamics has been appreciated \cite{breakspear2017dynamic} as a precursor for the application of control theory to cognitive neuroscience. 
	Among the paradigms in this regard are whole-brain Mesoscale Individualized NeuroDynamic (MINDy) models \cite{chen2024dynamical,singh2020estimation}. Our previous works have demonstrated the validity and utility of MINDy models as a generative tool for understanding the relationship between individualized neural architectures and neural dynamics in both resting-state \cite{singh2020estimation,singh2020scalable} and cognitive task contexts \cite{singh2022enhancing}. The MINDy models derived from single-subject resting-state brain imaging data operate at a macroscale level, where each node represents a distinct brain region. A given model captures the temporal evolution of brain activity and the non-linear interactions between hundreds of brain regions, taking the form of a nonlinear dynamical system of continuous-time Hopfield-type recurrent neural networks \cite{hopfield1984neurons}. All (intrinsic) parameters--the decay and connectivity matrices and the transfer function parameters--are directly estimated from brain activity time series so that the resulting models predict future brain activity, providing a more accurate and biologically plausible representation of large-scale brain dynamics.
	
	A MINDy model belongs to the more general class of Hopfield-type recurrent neural networks, which will constitute the focus of our study herein. Specifically, we consider a model of interconnected neural masses associated with $d$ regions, described by a set of nonlinear differential equations that represent the dynamics of each region's state over time. The dynamic for the $i$-th region in this network evolves as
	\begin{equation}\label{eq:Hopfield i-th neuron}
		\begin{split}
			\frac{dx_i}{dt}(t)&=-\alpha_i x_i(t)+\sum_{j=1}^{d}W_{ij}f_j(x_j(t))+\sum_{j=1}^{k}b_{ij}u_j\\
			x_i(0)&=x_i^0
		\end{split}
	\end{equation}
	where $k\in\N^{*}$ satisfies $k\le d$, $x_i(t)$ denotes the $i$-th region’s state at time $t\ge 0$, $\alpha_i>0$ is a positive parameter reflecting the rate at which the current state of the $i$-th region decays, $u_j$ are constant control inputs, $b_{ij}$ represent coupling coefficients between control inputs and neurons, 
	$W_{ij}$ is a real constant that weights the connection from the $j$-th region to the $i$-th region, and $f_j$ is the activation function of the  $j$-th region.

	The study of neural networks, such as \eqref{eq:Hopfield i-th neuron}, has long been a topic of great interest for researchers seeking to understand complex dynamical behaviors and their associated control mechanisms, \cite{albertini1993neural,forti1995new,hopfield1984neurons,sontag1998remarks,yi1999estimate,zhang2014comprehensive}. Among the various properties investigated, controllability--the ability to steer a network from one state to another within a finite time--stands out for its significant theoretical and practical implications. As highlighted earlier, in fields such as neuroengineering and brain stimulation, particularly transcranial electrical stimulation (tES) with direct current (stimulation via a constant input), the operative issue is to determine a constant input that can drive the state of the network \eqref{eq:Hopfield i-th neuron} to a desired target state--within the state space--over a short time horizon (Figure \ref{fig:scheme}). Such capability would present a promising avenue for developing stimulation protocols tailored to achieve specific therapeutic outcomes \cite{degiorgio2013neurostimulation}, such as altering the excitability of the motor cortex \cite{bergmann2009acute}.
	
	In previous works such as \cite{forti1995new,sontag1998remarks,yi1999estimate,zhang2014comprehensive}, controllability or stability properties of equations as \eqref{eq:Hopfield i-th neuron} have been considered in the full-actuated configuration when $d=k$ and $B=(b_{ij})=\idty$ or in the specific cases of $\alpha_i=0$ for all $i$, and $B$ belonging to an open dense subset of $\cM_{d,k}(\R)$.
	
	In the current study, we do not impose a priori such a restriction on the input matrix $B\in\cM_{d,k}(\R)$. The assumptions on the activation functions are relaxed to those commonly used in the literature in studying neural networks. The general consideration on $B$ and the complexity arising from the inherent nonlinear dynamics present significant challenges, and systematic solutions for the \textit{control synthesis} problem are not readily available in the current literature. To address this challenge, we derive an explicit algebraic representation for the hidden state $x(t)$ and the associated constant input that solves the two-point boundary value problem in state space up to higher-order terms. This yields an explicit small-time algebraic condition for the constant input, showing that the required control can be determined by inverting an expression involving the drift's Jacobian and the desired change in state, up to higher-order corrections with respect to the time horizon. From this derivation, we then provide a step function (piecewise constant control) for large periods of time. Although tDCS typically relies on constant or slowly modulated inputs, step constant controls provide a theoretically grounded and numerically efficient extension of constant-input synthesis over longer time horizons. Note that, in the case of linear activation functions, the synthesized constant control remains valid for both short and long time horizons.
	
	For input matrices that directly actuate $k$ nodes, we show that the reachable set with constant inputs is an affine subspace whose dimension equals the rank of $B$, and its basis can be computed explicitly via a thin QR factorization of the associated constraint matrix. This characterization applies to both linear and nonlinear activation functions, clarifying how the local structure of the input matrix shapes the short-horizon reachable set.

	The remainder of the paper is organized as follows. In the next section, we introduce the general notations used throughout the paper. Section~\ref{s:general notations} presents our modeling framework, reformulating equation~\eqref{eq:Hopfield i-th neuron} into a form more suitable for controllability analysis and establishing key results concerning solution representations. Section~\ref{s:exact controllability} addresses the control synthesis problem and is divided into several parts: Section~\ref{ss:controllability linear activation function} focuses on the linear activation case; Section~\ref{ss:controllability nonlinear activation function} develops control synthesis strategies for general nonlinear activations that appear to be the proper generalization of the linear case; and Section~\ref{ss:forward and Backward syntheses} outlines a taxonomy of synthesis approaches derived from dual representations of the solution, clarifying the conceptual origin of the forward nominal-state  and backward nominal-state strategies. Section~\ref{ss:reachability via control synthesis methods} then provides practical considerations by analyzing the structure of reachable sets associated with the proposed step-function control syntheses. Section~\ref{s:general comments} offers broader discussion and possible extensions. Section~\ref{s:examples and numerical simulation} presents numerical illustrations of the theory. Finally, Section~\ref{s:concluding remarks and discussion} summarizes the main results and discusses directions for future research. Technical proofs of some results are deferred to the Appendix.

	\begin{notations}
		
		In this paper, $\Z$ stands to the set of integers, $\N^{*}$ denote the set of positive integers and $\N=\N^{*}\cup\{0\}$. For all $a, b\in\N$, we denote by $\llbracket a,b\rrbracket=\N\cap[a,b]$. For all $d, k\in\N^{*}$, we denote by $\cM_{d,k}(\R)$ the set of $d\times k$ matrices with real coefficients.  $\R^d$ denotes the $d$-dimensional real column vector and $\cL(\R^d)$ the space of linear maps on $\R^d$ that we identify, in the usual way, with $\cM_d(\R)$ the set of squared matrices of order $d$ with real coefficients. For all $(x,y)\in(\R^d)^2$, we denote by $|x|$  the Euclidean norm of $x$ and by $\langle x,y\rangle$ the scalar product of $x$ and $y$. We denote the identity matrix in $\cM_d(\R)$ by $\idty$, and for every matrix $M\in\cM_{d,k}(\R)$, $\|M\|=\sup\{|Mx|\mid x\in\R^k,\;|x|=1\}$ denotes the spectral norm of $M$, and $M^\dagger$ denotes the Moore-Penrose pseudo-inverse of $M$. For a symmetric matrix $M\in\cM_d(\R)$, $\lambda_{\min}(M)$ and $\lambda_{\max}(M)$ denote respectively its smaller and largest eigenvalues. For ease of notation, a diagonal matrix $\Gamma=\diag(\gamma_1,\gamma_2,\cdots,\gamma_d)\in\cM_d(\R)$, will simply denoted as $\Gamma = (\gamma_i\delta_{i,j})$, where $\delta_{i,j}$  is the usual Kronecker symbol. If $A\in\cM_{d}(\R)$, then $e^{A}\in\cM_{d}(\R)$ and $\sigma(A)$ are, respectively, the exponential and the set of eigenvalues of $A$, viz.
		\begin{equation}\label{eq:exp(A)}
			e^A = \sum_{n=0}^\infty\frac{A^n}{n!},\quad\sigma(A)=\{\lambda\in\C\mid\ker(A-\lambda\idty)\neq 0\}.
		\end{equation}
		We recall that $\cO$ stands to the big $\cO$-notation, and $y(t)\underset{t \sim 0}{=} \cO(x(t))$ means $y(t)$ is of the order of $x(t)$. Finally, we use $\dot{x}(t)=\frac{dx}{dt}(t)$ for the total derivative of $x$ w.r.t. $t$.
	\end{notations}
	
	\section{Representation of solutions}\label{s:general notations}
	To study the controllability of~\eqref{eq:Hopfield i-th neuron}, it is convenient to recast it in abstract form. Introduce the (nonlinear) map
	\begin{equation}\label{eq:nonlinear vector field N}
		N(x) = -Dx+Wf(x)\qquad\forall x\in\R^d.
	\end{equation}
	
	Then, \eqref{eq:Hopfield i-th neuron} recasts as
	\begin{equation}\label{eq:Hopfield nonlinear}
		\dot{x}(t) = N(x(t))+Bu,\quad x(0) = x^0
	\end{equation}
	where $x=(x_1, x_2,\cdots, x_d)^\top\in\R^d$ is the state's vector, $x^0=(x_1^0,\cdots, x_d^0)^\top\in\R^d$ is the initial state, $u=(u_1,\cdots, u_k)^\top\in\R^k$ is a constant control, $B=(b_{ij})\in\cM_{d,k}(\R)$ is the input matrix, $D = (\alpha_i\delta_{i,j})\in\cM_d(\R)$ is the decay matrix, $W=(W_{ij})\in\cM_d(\R)$ is the connectivity matrix and $f:\R^d\to\R^d$ given by $f(x) = (f_1(x_1), f_2(x_2),\cdots, f_d(x_d))^\top$ is the firing rate function of the network. 
	
	\begin{notations}\label{nota:matrix A}
		We set $A=-D+W$, where $D$ is the decay matrix and $W$ is the connectivity matrix in \eqref{eq:nonlinear vector field N}.
	\end{notations}
	
	\begin{assumption}\label{ass:general assumption}
		
		Unless otherwise stated, we assume that for every $i\in\llbracket 1,d\rrbracket$, the activation function $f_i:\R\to\R$ is a $C^2$, globally Lipschitz function on $\R$ with a bounded second derivative, and satisfies $f_i(0)=0$. The latter is without loss of generality since we may always take for every $s\in\R$, $\tilde{f}_i(s)=f_i(s)-f_i(0)$ and set $\widetilde{Bu} = Bu+Wf(0)$ as the new input to equation \eqref{eq:Hopfield nonlinear}. Finally, for the sake of simplicity in the presentation, we also assume that $\|f_i'\|_\infty=1$. Indeed, as long as $\|f_i'\|_\infty\neq 0$, we can always define an activation function $\tilde{f}_i(s) = f_i(\lambda s)$ with $\lambda=1/\|f_i'\|_\infty$ and $s\in\R$. 
		
		Throughout the following, we will refer to the following fixed parameters:
		\begin{equation}
			\Gamma=\lambda_{\min}(D)-\|W\|,\quad \Lambda=\lambda_{\max}(D)+\|W\|, 
		\end{equation}
		\begin{equation}
			\Lambda_1=\|W\|\|f''\|_\infty\quad\text{and}\quad \|f''\|_\infty=\max_{i\in\llbracket 1,d\rrbracket}\|f_i''\|_\infty.
		\end{equation}
	\end{assumption}

	Recall that the vector field $N$ is globally Lipschitz on $\R^d$ (see, for instance, Lemma~\ref{lem:properties on N}). Denoting by $\phi_t:\R^d\to\R^d$ its flow at $t\in\R$, it follows that the family $\{\phi_t\mid t\in\R\}$ is a one parametric subgroup of $\operatorname{Diff(\R^d)}$, the group of diffeomorphism of $\R^d$. Note that for any $x^0\in \R^d$, $\phi_t(x^0)$ is the unique solution of \eqref{eq:Hopfield nonlinear} when the control $u=0$, namely
	\begin{equation}\label{eq::nonlinear flow}
		\frac{d}{d t}  \phi_t(x^0) = N(\phi_t(x^0)),\qquad
		\phi_0(x^0) = x^0.
	\end{equation}
	Let $\psi_t$ denote the inverse of $\phi_t$, i.e., $\psi_t=(\phi_t)^{-1}=\phi_{-t}$. Then, for every $x^0\in\R^d$, it holds
	\begin{equation}\label{eq::nonlinear flow psi}
		\frac{d}{d t} \psi_t(x^0) = -N(\psi_t(x^0)),\qquad
		\psi_0(x^0) = x^0.
	\end{equation}
	In particular, the map $(t, x)\in\R\times\R^d\mapsto\phi_t(x)\in\R^d$ is of class $C^2$ (see, for instance, \cite[Chapter~15, Theorem~1]{hirsch1974differential}). Moreover, for a fixed $t\in\R$, if we use $D\phi_t(x)$ to denote the differential of $\phi_t$, then $D\phi_t(x)$ is a well-defined invertible matrix, and it holds
	\begin{equation}\label{inverse of the differential of the flow}
		\left[D\phi_t(x)\right]^{-1} = D\psi_t(\phi_t(x)),\qquad\forall x\in\R^d.
	\end{equation}
	
	The classical theory of ordinary differential equations (ODE) can be invoked to justify the existence and uniqueness of solutions of \eqref{eq:Hopfield nonlinear} (see, for instance, \cite[Chapter 2]{bressan2007introduction}). In this work, we represent the system's state trajectory in a form reminiscent of that of linear time-invariant systems as used in \cite[Theorem 4.4]{tamekue2024mathematical}. This representation emerged from the chronological calculus framework introduced by Agrachev and Gamkrelidze in the late 1970s for solving nonautonomous ODEs on finite-dimensional manifolds \cite{agravcev1979exponential}. For a comprehensive treatment of this theory, see also \cite[Chapter~6]{agrachev2019comprehensive}, and \cite{kawski2011chronological,sarychev2006lie} for applications in geometric control theory. 
	This representation is noteworthy because, in the specific case of linear activation functions, it naturally aligns with the variation of the constants formula. From this perspective, the proposed representation and approach are a proper generalization of the variation of constants formula to the case of models of the form \eqref{eq:Hopfield nonlinear}.

	
	The first result of this paper is about solution representation, and it states as follows. The proof is presented in Section~\ref{ss:proof of representation of solution}.
	\begin{theorem}\label{thm:backward representation}
		Let $T>0$ and $(x^0, u)\in\R^d\times\R^k$. The solution $x\in C^3([0, T];\R^d)$ of \eqref{eq:Hopfield nonlinear} can be expressed  as
		\begin{equation}\label{eq:backward representation}
			x(t) = \phi_t\left(x^0+\int_{0}^{t}D\psi_s(x(s))Bu\,ds\right),\quad\forall t \in [0, T].
		\end{equation}
	\end{theorem}
	
	When the dynamics are linear, this representation reduces to the familiar form. 
	
	\begin{corollary}\label{cor:solution in the linear case}
		Let $T>0$ and $(x^0, u)\in\R^d\times\R^k$. If $f_i(s)=s$, the solution $x\in C^\infty([0, T];\R^d)$ to~\eqref{eq:Hopfield nonlinear} reads
		\begin{equation}\label{eq:solution linear}
			x(t) = e^{tA}x^0+\int_0^te^{(t-s)A}Bu\,ds\qquad\forall t\in [0, T].
		\end{equation}
	\end{corollary}
	
	\begin{proof}
		In this case, $N(x)=Ax$ for every $x\in\R^d$ where $A$ is introduced in Notation~\ref{nota:matrix A}. In this case, $\phi_t(x)=e^{tA}x$ and $\psi_t(x)=e^{-tA}x$  for every $t\in\R$. It follows that $D\psi_t(x(t))=e^{-tA}$ and the result then follows by~\eqref{eq:backward representation}. 
	\end{proof}
	
	We conclude this section by presenting a second representation of the solution to \eqref{eq:Hopfield nonlinear}. It
	can be viewed as a natural extension of the linear case given in Corollary~\ref{cor:solution in the linear case}, particularly at the final time horizon $T$. The proof follows similar lines to that of Theorem~\ref{thm:backward representation}.
	
	\begin{theorem}\label{thm:forward representation}
		Let $T > 0$ and $(x^0, u) \in \R^d \times \R^k$. Then the solution $x \in C^3([0, T]; \R^d)$ of \eqref{eq:Hopfield nonlinear} can be expressed as
		\begin{equation}\label{eq:forward representation}
			x(t) = \psi_{T-t} \left( \phi_T(x^0) + \int_{0}^{t} D\phi_{T-s}(x(s)) B u \, ds \right)
		\end{equation}
		for all $t \in [0, T] $.
	\end{theorem}
	
	\begin{remark}\label{rmk:on piecewise control}
		Although Theorems~\ref{thm:forward representation} and~\ref{thm:backward representation} are stated for constant inputs $u \in \mathbb{R}^k$, they remain valid for time-dependent inputs $u \in L^1([0, T]; \mathbb{R}^k)$. Note that this is a specific case of \cite[Theorem 3.2]{tamekue2025control}. In this case, the corresponding solution $x(\cdot)$ to~\eqref{eq:Hopfield nonlinear} belongs to $C^0([0, T]; \mathbb{R}^d)$. This observation is particularly relevant, as we will later consider step constant controls---i.e., piecewise constant controls defined on $[0, T]$.
	\end{remark}
	
	\section{Controllability of the neural networks with step function}\label{s:exact controllability}
	In this section, we analyze the controllability of~\eqref{eq:Hopfield nonlinear} under constant inputs over short time horizons and extend the approach to step-function controls for longer-horizon transfers in the fully nonlinear setting. When the activation functions are linear, the synthesized constant control remains valid for both short and long time horizons.
	
	Recall, e.g., from \cite[Remark~4.2.3]{tucsnak2009observation} that a step function on $[0, T]$ is a piecewise constant function defined over $[0, T]$, i.e., constant on each interval of a partition of $[0, T]$. Then, a constant function corresponds to a special case of a step function with a single constant value across the entire interval.
	
	Let us introduce the following.
	
	\begin{definition}[Step controllability]\label{def::Linear control notion}
		Let $T>0$.  System~\eqref{eq:Hopfield nonlinear} is step controllable over the time interval $[0, T]$ if, for all $x^0, x^1\in \R^d$, there exists a step function $u$ on $[0, T]$ such that the solution of \eqref{eq:Hopfield nonlinear} with $x(0)=x^0$ satisfies $x(T) = x^1$.
	\end{definition}

	\subsection{The case of linear activation functions}\label{ss:controllability linear activation function}
	
	To gain insight into the controllability of the nonlinear system~\eqref{eq:Hopfield nonlinear} under constant control, it is instructive to first analyze the corresponding linearized model---a strategy that has been clearly articulated and effectively motivated in prior work, notably by \cite{liu2011controllability}.
	
	Consider first the linear neural network
	\begin{equation}\label{eq:linear model}
		\dot{x}(t) = Ax(t) + Bu, \qquad x(0) = x^0 \in \mathbb{R}^d,
	\end{equation}
	with $A = -D + W$ (corresponding to $f_i(s) = s$ for all $i$) and constant input $u \in \mathbb{R}^k$.  System \eqref{eq:linear model} is controllable on $[0,T]$ if and only if the Kalman rank condition holds.
	Equivalently, the invertibility of the controllability Gramian
	also guarantees controllability over $[0,T]$. Moreover, when these conditions are met, one can synthesize a time-varying control $\underline{u}\in L^\infty((0,T);\mathbb{R}^k)$ of minimal $L^2$-energy that steers \eqref{eq:linear model} from any $x^0$ to any $x^1$ in time $T$ (see, e.g., \cite{coron2007control,tucsnak2009observation}).
	
	However, this synthesis does not systematically extend to \textit{constant control} capable of achieving the same objective. If the system is controllable, and we assume the existence of a constant input $\tilde{u} \in \mathbb{R}^k$ that steers the system from $x^0 \in \mathbb{R}^d$ to $x^1 \in \mathbb{R}^d$ over the interval $[0, T]$, then { from \eqref{eq:solution linear} and $x(T)=x^1$, 
		\[
		\int_0^Te^{(T-s)A}\,dsB\tilde{u} = x^1 - e^{T A} x^0
		\]
		so that left multiplying by $A$ and using the identity
		\begin{equation}\label{eq:trick}
			A \int_0^T e^{(T-s)A} \, ds=\int_0^T\hspace{-0.2cm}\left(-\frac{d}{ds} e^{(T-s)A} \right) ds =  e^{T A} - \idty
		\end{equation}
		we find
	}
	\begin{equation}\label{eq:a prior analysis}
		\left(e^{T A} - \idty\right) B \tilde{u} = A \left(x^1 - e^{T A} x^0 \right).
	\end{equation}
	However, \eqref{eq:a prior analysis} can be ill-posed, as shown by the following counter-example.
		\begin{remark}\label{rmk:a priori analysis}
			Consider the case where $k = d = 2$, $T = 2\pi/\omega$, $B = \idty$, and
			\[
			A = 
			\begin{bmatrix}
				0 & -\omega \\
				\omega & 0
			\end{bmatrix},\quad
			e^{TA} = 
			\begin{bmatrix}
				\cos(2\pi) & -\sin(2\pi) \\
				\sin(2\pi) & \cos(2\pi)
			\end{bmatrix} = \idty.
			\]
			The linear system \eqref{eq:linear model}, equipped with time-varying control, is fully actuated and hence controllable, as guaranteed by the Kalman rank condition. However, any constant control $\tilde{u} \in \mathbb{R}^2$ must satisfy \eqref{eq:a prior analysis}. Since $e^{TA} - \idty = 0$, it follows that no constant control $\tilde{u} \in \mathbb{R}^2$ can steer $x^0 = 0$ to a nonzero target state $x^1 \neq 0$ over the interval $[0, T]$. Notably, $T = 2\pi/\omega$ can be made arbitrarily small by choosing $\omega \gg 2\pi$.
			
			On the other hand, using the Gramian-based synthesis, one finds that $W_c^{-1} = \idty/T$, and
			\begin{equation}
				\underline{u}(t) = \frac{1}{T}e^{(T-t)A^\top}x^1,\qquad t \in [0, T]
			\end{equation}
			is the time-varying control of minimal $L^2$-norm that steers $x^0 = 0$ to $x^1 \neq 0$. In particular, $\|\underline{u}\|_{L^2}^2 = \langle W_c^{-1}x^1, x^1 \rangle = |x^1|^2/T$. Taking $x^1 = (\sqrt{T/\omega}, 0)^\top$ yields $\|\underline{u}\|_{L^2}^2 = 1/\omega$, showing that if $\omega \gg 1$, then $\underline{u}$ achieves the transfer from $x^0 = 0$ to $x^1 = (\sqrt{T/\omega}, 0)^\top$ in $T > 0$ with small effort.
		\end{remark}
	
	Remark~\ref{rmk:a priori analysis} suggests that the controllability of a linear system does not imply the existence of a constant control that solves the control objective.
	
	\begin{proposition}\label{pro:controllability in the linear case}
			Let $T > 0$ and $(x^0, x^1) \in(\R^d)^2$. Assume that for any $\ell \in\Z$, 
			$i \frac{2\pi \ell}{T}\not\in\sigma(A)$. Then, the solution $x(\cdot)$ to \eqref{eq:linear model} satisfies $x(T)=x^1$, if and only if, $u\in\R^k$ solves
			\begin{equation}\label{eq:linear control particular}
				Bu = \left(e^{T A}-\idty\right)^{-1}A\left(x^1-e^{TA}x^0\right).
			\end{equation}
			
	\end{proposition}
	The proof of the following result is immediate.
		\begin{corollary}\label{cor:necessary and sufficient cond for solvability linear}
			Under hypotheses of Proposition~\ref{pro:controllability in the linear case}, \eqref{eq:linear control particular} has at least one solution $u\in\R^k$, if and only if, 
			\begin{equation}
				\left(e^{T A}-\idty\right)^{-1}A\left(x^1-e^{TA}x^0\right)\in\im B.
			\end{equation}
			In this case, $u\in\R^k$ given by
			\begin{equation}
				u = B^\dagger\,\left(e^{T A}-\idty\right)^{-1}A\left(x^1-e^{TA}x^0\right)
			\end{equation}
			is the corresponding least-norm constant control.
		\end{corollary}
	\begin{proof}[Proof of Proposition~\ref{pro:controllability in the linear case}] 
		Under the assumption on the eigenvalues of $A$, the matrices $A$ and $e^{TA}-\idty$ are invertible. By Corollary~\ref{cor:solution in the linear case}, the solution $x\in C^\infty([0, T];\R^d)$ of \eqref{eq:linear model} with a constant control $u\in\R^k$ is given by \eqref{eq:solution linear}. Let us show that $x(T)=x^1$ is equivalent to \eqref{eq:linear control particular}. If $x(T)=x^1$, then one obtains immediately that $u\in\R^k$ solves \eqref{eq:linear control particular}
		by invoking the same arguments used to derive \eqref{eq:a prior analysis}. Conversely, if \eqref{eq:linear control particular} is satisfied, 
		one deduces from \eqref{eq:trick} and \eqref{eq:solution linear} that
		\begin{equation}
			Ax(T) = Ae^{TA}x^0+\left(e^{T A}-\idty\right)Bu=Ax^1
		\end{equation}
		so that $x(T)=x^1$ since $A$ is invertible. 
	\end{proof}
	\begin{remark}\label{rmk:linear case}
			If $\|W\|<\lambda_{\min}(D)$, then the matrix $e^{TA}-\idty$ is invertible for all $T>0$, as ensured by Lemma~\ref{lem:dissipativity of A} and the Neumann series expansion. Moreover, for all $x\in\R^d$, we have
			\[
			|A x| \ge |D x| - |W x| \ge \left( \lambda_{\min}(D) - \|W\| \right) |x|
			\]
			which implies that $A$ is invertible as well.
			
			Therefore, in Proposition~\ref{pro:controllability in the linear case}, the spectral assumptions 
			become relevant only when $\|W\| \ge \lambda_{\min}(D)$.
	\end{remark}

	\begin{remark}
		\label{rmk:constant control synthesis}
		Remark~\ref{rmk:a priori analysis} and Proposition~\ref{pro:controllability in the linear case} highlight a key distinction between constant and time-varying control synthesis in linear systems. 
		
		While time-varying control synthesis depends solely on the structural properties of the pair $(A, B)$--as captured by the invertibility of the controllability Gramian---constant control synthesis is more restrictive. In particular, it depends not only on the matrices $A$ and $B$, but also crucially on the time horizon $T>0$, the initial and target states $x^0$ and $x^1$, and on the spectrum of $A$. As shown in Remark~\ref{rmk:a priori analysis}, even a fully actuated system may fail to connect certain states under constant input when $e^{TA} = \idty$, rendering \eqref{eq:a prior analysis}, then \eqref{eq:linear control particular} unsolvable.
	\end{remark}
	
	\subsection{The case of nonlinear activation functions}\label{ss:controllability nonlinear activation function}
	
	Proposition~\ref{pro:controllability in the linear case} highlights that controllability under constant control depends intricately on the time horizon~$T$, the spectrum of $DN(\cdot)$, and the initial and target states $x^0$ and $x^1$.
	
	Based on the solution representations of \eqref{eq:Hopfield nonlinear} given in Theorems~\ref{thm:backward representation} and~\ref{thm:forward representation}, we note that the latter naturally extends the linear case, particularly at the final time horizon \( T \). Nonetheless, both formulations provide valid foundations for synthesizing controls that achieve the desired state transfer. The key distinction lies in their temporal structure and directionality, as well as in the specific conditions required to guarantee their applicability.
	
	\subsubsection{Forward nominal-state synthesis}\label{sss:Forward nominal-state synthesis}
	Building on the solution representation of \eqref{eq:Hopfield nonlinear} given in Theorem~\ref{thm:forward representation}, we aim to synthesize step-function controls that achieve the control objective. Unlike the linear case, the control is first derived in an implicit form due to the nonlinear nature of the system. Then, under the assumption of a small time horizon, we provide an explicit expression that is an accurate approximation of the step-function controls that solve the control objective.
	
	We begin by stating a key result on which the constant control synthesis is based. The proof is identical to that presented in Section~\ref{s:smooth and Lipschitz activation function}.
	
	\begin{proposition}\label{pro:smooth and Lipschitz activation function 1} 
		Let $T>0$ and $(x_0, u)\in\R^d\times\R^k$. The solution $x(\cdot)$ of \eqref{eq:Hopfield nonlinear} can be expanded  at time $t=T$ as 
		\begin{IEEEeqnarray}{rCl}\label{eq:sol expansion with remainder 1}
			x(T)&=&\sum_{n=0}^{\infty}(-1)^n\frac{T^{n+1}DN\left(\phi_T(x^0)\right)^n}{(n+1)!}D\phi_T(x^0)Bu\nonumber\\
			&&+\phi_T(x^0)-\varphi_u(T)Bu.
		\end{IEEEeqnarray}
		Here $\varphi_u(T)=\kappa_u(T)+\eta_u(T)$,
		\begin{equation}\label{eq:kappa_u}
			\kappa_u(T) = \sum_{n=1}^{\infty}\int_{0}^{T}\frac{(t-T)^{n+1}}{(n+1)!}\left[\frac{d}{dt}Z_{T-t}^n\right]Q_{T-t}\,dt,
		\end{equation}
		\begin{equation}\label{eq:eta_u}
			\hspace{-0.3cm}\eta_u(T)= \sum_{n=1}^{\infty}\int_{0}^{T}\frac{(t-T)^nZ_{T-t}^{n-1}}{n!}D^2\phi_{T-t}(x(t))\dot{x}(t)\,dt
		\end{equation}
		where $Q_{T-t}=D\phi_{T-t}(x(t))$, $Z_{T-t}=DN(\phi_{T-t}(x(t)))$ and $D^2\phi_{T-t}(x(t))$ is the second derivative of $\phi_{T-t}$ at $x(t)$.
	\end{proposition}
	
	The implicit constant control synthesis of this section is therefore presented in the following theorem.
	
	\begin{theorem}\label{thm:Forward nominal-state synthesis}
		Let $T > 0$ and $(x^0, x^1)\in(\R^d)^2$. Assume that for any $\ell\in\Z$, 
		$i \frac{2\pi \ell}{T}\not\in\sigma(DN(\phi_T(x^0)))$. Then, the solution $x(\cdot)$ to \eqref{eq:Hopfield nonlinear} satisfies $x(T)=x^1$, if and only if, $u\in\R^k$ solves
		\begin{equation}\label{eq:forward nominal-state synthesis}
			\left[\idty-\cA_T(x^0)\varphi_u(T)\right]Bu=\cA_T(x^0)(x^1-\phi_T(x^0)).
		\end{equation}
		Here $\varphi_u(T)\in\cM_d(\R)$ is defined in Proposition~\ref{pro:smooth and Lipschitz activation function 1}, and letting $U_T(x^0)=DN(\phi_T(x^0))$, one has
		\begin{equation}\label{eq:matrix A_T(x^0)}
			\hspace{-0.25cm}\cA_T(x^0) = D\psi_T(\phi_T(x^0))\left(\idty-e^{-TU_T(x^0)} \right)^{-1}U_T(x^0).
		\end{equation}
	\end{theorem}

	\begin{proof}
		Let us show that $x(T)=x^1$ is equivalent to \eqref{eq:forward nominal-state synthesis}. If $x(T)=x^1$, then left multiplying \eqref{eq:sol expansion with remainder 1} by $U_T(x^0)$, using~\eqref{eq:exp(A)}, and
		\begin{equation}
			\sum_{n=0}^{\infty}\frac{(-TU_T(x^0))^{n+1}}{(n+1)!}= \sum_{n=1}^{\infty}\frac{(-TU_T(x^0))^{n}}{n!}=e^{-TU_T(x^0)}-\idty
		\end{equation}
		yields
		\begin{IEEEeqnarray}{rCl}\label{eq:derivation of the nonlinear input 1-2}
			U_T(x^0)\left(x^1-\phi_T(x^0)\right)&=&-U_T(x^0)\varphi_u(T)Bu\nonumber\\
			&&\hspace{-1.5cm}+(\idty-e^{-TU_T(x^0)})D\phi_T(x^0)Bu.
		\end{IEEEeqnarray}
		Since $[D\phi_T(x^0)]^{-1}=D\psi_T(\phi_T(x^0))$, we left multiply \eqref{eq:derivation of the nonlinear input 1-2} respectively by $(\idty-e^{-TU_T(x^0)})^{-1}$ and $D\psi_T(\phi_T(x^0))$ to get
		\begin{equation}\label{eq:solving 1-0}
			\cA_T(x^0)\left(x^1-\phi_T(x^0)\right)=\left[\idty-\cA_T(x^0)\zeta_u(T)\right]Bu
		\end{equation}
		which is exactly \eqref{eq:forward nominal-state synthesis} where $\cA_T(x^0)$ is defined by \eqref{eq:matrix A_T(x^0)}. Conversely, assume that $u\in\R^k$ solves \eqref{eq:forward nominal-state synthesis}. Then, one finds
		\begin{equation}\label{eq:solving 1-1}
			\cA_T(x^0)\varphi_u(T)Bu = Bu-\cA_T(x^0)\left(x^1-\phi_T(x^0)\right).
		\end{equation}
		One the other hand, from \eqref{eq:sol expansion with remainder 1}, one deduces that
		\begin{equation}\label{eq:solving 2-2}
			\cA_T(x^0)\varphi_u(T)Bu = Bu-\cA_T(x^0)\left(x(T)-\phi_T(x^0)\right).
		\end{equation}
		Identifying \eqref{eq:solving 1-1} with \eqref{eq:solving 2-2}, one gets
		\begin{equation}\label{eq:equation satisfying x(T) and x1}
			\cA_T(x^0)(x(T)-x^1)=0.
		\end{equation}
		Since $\cA_T(x^0)$ is invertible, it follows from \eqref{eq:equation satisfying x(T) and x1} that $x(T)=x^1$. This completes the proof of the necessary part. 
	\end{proof}
	
	The synthesis provided in Theorem~\ref{thm:Forward nominal-state synthesis} is implicit, as the operator $\varphi_u(T)$ depends on the state trajectory $x(\cdot)$, and thus on the control input $u$. This dependence poses practical limitations, even when a solution exists. In applications such as brain stimulation via tDCS, one is often interested in modulating activity over short time horizons. In what follows, we derive an expansion of the matrix $\varphi_u(T)$ in the small-time regime that enables accurate synthesis in small time horizons.
	
	The proof of the following is presented in Section~\ref{ss:spectral norm of varphi_u(T)}.
	
	\begin{proposition}\label{pro:control synthesis in small time 1}
		Under the assumptions of Theorem~\ref{thm:Forward nominal-state synthesis}, the following expansion holds
		\begin{IEEEeqnarray}{rCl}\label{eq:expansion of varphi_u}
			\varphi_u(T) &\underset{T \sim 0}{=}& \left(\idty - e^{-T U_T(x^0)}\right) U_T(x^0)^{-1} D\phi_T(x^0) \nonumber\\
			&&- \left(e^{T U_T(x^0)} - \idty\right) U_T(x^0)^{-1} + \cO(T^3)
		\end{IEEEeqnarray}
		In particular, the following expansion also holds
		\begin{equation}\label{eq:A_T and varphi_u}
			\hspace{-0.2cm}\idty - \mathcal{A}_T(x^0) \varphi_u(T) \underset{T \sim 0}{=} D\phi_T(x^0)^{-1} e^{-T U_T(x^0)} + \cO(T^2).
		\end{equation}
	\end{proposition}
	
	\begin{remark}\label{rmk:on big O}
		It is worth noting that the $\cO(T^3)$ term in~\eqref{eq:expansion of varphi_u} becomes $\cO(T^2)$ in~\eqref{eq:A_T and varphi_u} due to the relation
		\begin{equation}
			\mathcal{A}_T(x^0)\, \cO(T^3) \underset{T \sim 0}{=} \cO(T^2),
		\end{equation}
		since $\mathcal{A}_T(x^0)$ includes the factor $\left(\idty - e^{-T U_T(x^0)}\right)^{-1}$, which behaves like $\cO(T^{-1})$ as $T \to 0$, given that $U_T(x^0)$ is uniformly bounded with respect to $T>0$ and $x^0$.
	\end{remark}

	The explicit constant control synthesis of this section is summarized in the following theorem, which includes an estimate of the endpoint error.
	
	\begin{theorem}\label{thm:endpoint error of the forward nominal-state synthesis control}
		Let $T > 0$ and $(x^0, x^1)\in(\R^d)^2$. Assume that for any $\ell\in\Z$, 
		$i \frac{2\pi \ell}{T}\not\in\sigma(DN(\phi_T(x^0)))$. Then, the solution $x(\cdot)$ to \eqref{eq:Hopfield nonlinear} satisfies $x(T)=x^1$, if and only if, $u\in\R^k$ solves
		\begin{equation}\label{eq:equation for u general}
			Bu \underset{T \sim 0}{=} \left[e^{T U_T(x^0)} - \idty\right]^{-1} U_T(x^0) \left(x^1 - \phi_T(x^0)\right)+\cO(T^2)
		\end{equation}
		where $U_T(x^0) := DN(\phi_T(x^0))$. In particular, let $\tilde{x}(\cdot)$ denote the solution of~\eqref{eq:Hopfield nonlinear} corresponding to  $\tilde{u} \in \mathbb{R}^k$ satisfying
		\begin{equation}\label{eq:approximation of the forward nominal-state synthesis control}
			B \tilde{u} = \left[e^{T U_T(x^0)} - \idty\right]^{-1} U_T(x^0) \left(x^1 - \phi_T(x^0)\right).
		\end{equation}
		Then, the following endpoint error estimate holds
		\begin{equation}\label{eq:endpoints error forward nominal-state synthesis control}
			|\tilde{x}(T) - x^1| \underset{T \sim 0}{=} \cO(T^2).
		\end{equation}
	\end{theorem}
	
	\begin{proof} First, \eqref{eq:equation for u general} is an immediate consequence of~\eqref{eq:forward nominal-state synthesis}, \eqref{eq:matrix A_T(x^0)} and~\eqref{eq:A_T and varphi_u}. Next, using the expansions~\eqref{eq:sol expansion with remainder 1}, \eqref{eq:exp(A)}, \eqref{eq:expansion of varphi_u}, and the definition of $\tilde{u}$ from~\eqref{eq:approximation of the forward nominal-state synthesis control}, we obtain
		\begin{IEEEeqnarray}{rCl}\label{eq:sol expansion with remainder 1 tilde u}
			\tilde{x}(T) &=& \phi_T(x^0) + \sum_{n=0}^{\infty} (-1)^n \frac{T^{n+1} U_T(x^0)^n}{(n+1)!} D\phi_T(x^0) B\tilde{u} \nonumber\\
			&& - \varphi_u(T) B\tilde{u} \nonumber\\
			&\underset{T \sim 0}{=}& \left[\idty - e^{-T U_T(x^0)}\right] U_T(x^0)^{-1} D\phi_T(x^0) B\tilde{u} \nonumber\\
			&& - \left[\idty - e^{-T U_T(x^0)}\right] U_T(x^0)^{-1} D\phi_T(x^0) B\tilde{u} \nonumber\\
			&& - \left(e^{T U_T(x^0)} - \idty\right) U_T(x^0)^{-1} B\tilde{u} + \phi_T(x^0)\nonumber\\
			&&+ \cO(T^2) \underset{T \sim 0}{=} x^1 + \cO(T^2).
		\end{IEEEeqnarray}
		The same observation as in Remark~\ref{rmk:on big O} justifies why the $\cO(T^3)$ term in~\eqref{eq:expansion of varphi_u} reduces to $\cO(T^2)$ in~\eqref{eq:sol expansion with remainder 1 tilde u}. 
	\end{proof}
	
	The following result is immediate. It provides a necessary and sufficient condition for \eqref{eq:approximation of the forward nominal-state synthesis control} to admit at least one solution.
	
	\begin{corollary}\label{cor:necessary and sufficient cond for solvability nonlinear forward}
		Under the hypotheses of Theorem~\ref{thm:Forward nominal-state synthesis}, \eqref{eq:approximation of the forward nominal-state synthesis control} admits at least one solution $\tilde{u} \in \mathbb{R}^k$ if and only if
		\begin{equation}
			\left[e^{T U_T(x^0)} - \idty\right]^{-1} U_T(x^0)\left(x^1 - \phi_T(x^0)\right) \in \operatorname{Im} B.
		\end{equation}
		In this case, the least-norm constant control is given by
		\begin{equation}
			\tilde{u} = B^\dagger\, \left[e^{T U_T(x^0)} - \idty\right]^{-1} U_T(x^0)\left(x^1 - \phi_T(x^0)\right).
		\end{equation}
	\end{corollary}

	While the previous synthesis applies in a short time, some practical scenarios require reaching the target state over a longer time. In the following, we show how a step function can be constructed from a short-time synthesis to achieve state transfer in a long-time horizon. As noted in Remark~\ref{rmk:on piecewise control}, the solution to \eqref{eq:Hopfield nonlinear} can still be represented by \eqref{eq:forward representation}. 
	
	We begin with the following implicit synthesis for step-function controls. Its proof, outlined in Section~\ref{ss:proof of forward implicit nominal-state synthesis step function}, closely follows the arguments in Proposition~\ref{pro:smooth and Lipschitz activation function 1} and Theorem~\ref{thm:Forward nominal-state synthesis}.
	
	\begin{theorem}\label{thm:forward implicit nominal-state synthesis step function}
		Let $(x^0, x^1)\in(\R^d)^2$ and $T\ge\tau$, where $\tau>0$ be such that for any $\ell\in\Z$, 
		$i \frac{2\pi \ell}{\tau}\not\in\sigma(DN(\phi_T(x^0)))$. Define the step-function $u_{sf, \tau}:[0, T]\to\R^k$,
		\begin{equation}\label{eq:step-function control}
			u_{sf, \tau}(t) = \begin{cases}
				0&\quad\mbox{if}\quad 0\le t\le T-\tau\\
				u&\quad\mbox{if}\quad T-\tau< t\le T
			\end{cases}
		\end{equation}
		where $u\in\R^k$. Then, the solution $x_\tau(\cdot)$ to~\eqref{eq:Hopfield nonlinear} corresponding to $u_{sf, \tau}(\cdot)$ satisfies $x_\tau(T)=x^1$, if and only if, $u\in\R^k$ solves
		\begin{equation}\label{eq:forward implicit nominal-state synthesis step function}
			\left[\idty-\cA_{\tau,T}(x^0)\varphi_{u}(\tau,T)\right]Bu=\cA_{\tau,T}(x^0)\left(x^1-\phi_T(x^0)\right).
		\end{equation}
		Here,  $U_T(x^0)=DN(\phi_T(x^0))$ and
		\begin{equation}
			\cA_{\tau,T}(x^0)= D\psi_{\tau}(\phi_T(x^0))\left[\idty-e^{-\tau U_T(x^0)}\right]^{-1}U_T(x^0).
		\end{equation}
		Moreover, $\varphi_{u}(\tau,T) = \kappa_{u}(\tau,T) + \eta_{u}(\tau,T)$, with $\kappa_{u}(\tau,T)$ and $\eta_{u}(\tau,T)$ defined in \eqref{eq:kappa_u} and \eqref{eq:eta_u}, respectively, where the integrals are taken over $[T - \tau, T]$.
	\end{theorem}
	Then, one has the following synthesis in a long-time horizon using explicit step-function controls.
	
	\begin{theorem}\label{thm:forward explicit nominal-state synthesis step function}
		Let $(x^0, x^1)\in(\R^d)^2$ and $T\ge\tau$, where $\tau>0$ be such that for any $\ell\in\Z$, 
		$i \frac{2\pi \ell}{\tau}\not\in\sigma(DN(\phi_T(x^0)))$. Let $\tilde{u}_\tau\in\R^k$ be a solution of 
		\begin{equation}
			B\tilde{u}_\tau =  \left[e^{\tau U_T(x^0)} - \idty\right]^{-1} U_T(x^0)\left(x^1 - \phi_T(x^0)\right)
		\end{equation}
		where $U_T(x^0):=DN(\phi_T(x^0))$. Then, the solution $\tilde{x}_\tau(\cdot)$ to~\eqref{eq:Hopfield nonlinear} corresponding to the step function $\tilde{u}_{sf, \tau}:[0,T]\to\R^k$,
		\begin{equation}
			\tilde{u}_{sf, \tau}(t) = \begin{cases}
				0&\quad\mbox{if}\quad 0\le t\le T-\tau\\
				\tilde{u}_\tau&\quad\mbox{if}\quad T-\tau< t\le T
			\end{cases}
		\end{equation}
		satisfies 
		\begin{equation}\label{eq:endpoints error forward nominal-state synthesis step function}
			|\tilde{x}_\tau(T)-x^1|\underset{\tau \sim 0}{=} \cO(\tau^2).
		\end{equation}
	\end{theorem}
	
	The proof of Theorem~\ref{thm:forward explicit nominal-state synthesis step function} follows the same arguments as Theorem~\ref{thm:endpoint error of the forward nominal-state synthesis control} and is therefore omitted for brevity.
	
	\subsubsection{Backward nominal-state synthesis}\label{sss:Backward nominal-state synthesis}
	
	Building on the solution representation of \eqref{eq:Hopfield nonlinear} given in Theorem~\ref{thm:forward representation}, we aim to construct step-function controls under the assumption that the terminal condition $x(T) = x^1$ is achieved.
	In contrast to the forward nominal-state synthesis, the implicit control derived here does not provide a direct condition guaranteeing that the target state is reached.
	
	We begin with the key result on which the implicit synthesis is based. The proof is presented in Section~\ref{s:smooth and Lipschitz activation function}.
	\begin{proposition}\label{pro:smooth and Lipschitz activation function}
			Let $T>0$ and $(x_0, u)\in\R^d\times\R^k$. The solution $x(\cdot)$ of \eqref{eq:Hopfield nonlinear} can be expanded  as
			\begin{equation}\label{eq:sol expansion with remainder}
				x(t) = \phi_t\left(x^0+\sum_{n=0}^{\infty}\frac{t^{n+1}Z_t^n}{(n+1)!}P_tBu-\zeta_u(t)Bu\right)
			\end{equation}
			for all $t\in [0, T]$. Here $\zeta_u(t)=\xi_u(t)+\chi_u(t)$,
			\begin{equation}\label{eq:xi}
				\xi_u(t) = \sum_{n=1}^{\infty}\int_{0}^{t}\frac{s^{n+1}}{(n+1)!}\left[\frac{d}{ds}Z_s^n\right]P_sds,
			\end{equation}
			\begin{equation}\label{eq:chi}
				\chi_u(t)= \sum_{n=1}^{\infty}\int_{0}^{t}\frac{s^n}{n!}Z_s^{n-1}D^2\psi_s(x(s))\dot{x}(s)ds,
			\end{equation}
			where $P_t=D\psi_t(x(t))$, $Z_t=DN(\psi_t(x(t)))$ and $D^2\psi_s(x(s))$ is the second derivative of $\psi_s$ at $x(s)$.
			
	\end{proposition}

	\begin{theorem}\label{thm:backward nominal-state synthesis}
			Let $T > 0$ and $(x^0, x^1)\in(\R^d)^2$. Assume that for any $\ell\in\Z$, 
			$i \frac{2\pi \ell}{T}\not\in\sigma(DN(\psi_T(x^1)))$. Any constant control $u \in \mathbb{R}^k$ that ensures the corresponding solution $x(\cdot)$ of \eqref{eq:Hopfield nonlinear} satisfies $x(T) = x^1$ necessarily solve
			\begin{equation}\label{eq:backward nominal-state synthesis}
				\left(\idty - \cB_T(x^1)\zeta_u(T)\right)B u = \cB_T(x^1)\left(\psi_T(x^1) - x^0\right).
			\end{equation}
			Here $\zeta_u(T)\in\cM_d(\R)$ is defined in Proposition~\ref{pro:smooth and Lipschitz activation function}, and
			letting $V_T(x^1)=DN(\psi_T(x^1))$, one has
			\begin{equation}\label{eq:matrix B_T(x^1)}
				\cB_T(x^1) = D\phi_T(\psi_T(x^1))\left(e^{TV_T(x^1)}-\idty\right)^{-1}V_T(x^1).
			\end{equation}
	\end{theorem}
	
	\begin{remark}
		The implicit backward nominal-state synthesis in Theorem~\ref{thm:backward nominal-state synthesis} provides only a necessary condition for a constant control $u \in \mathbb{R}^k$ to achieve $x(T) = x^1$ in~\eqref{eq:Hopfield nonlinear}. In contrast, the forward synthesis~\eqref{eq:forward nominal-state synthesis} yields a sufficient condition: any $u$ solving it guarantees $x(T) = x^1$, and thus also satisfies the backward condition.
	\end{remark}
	
	Let us now present the proof of Theorem~\ref{thm:backward nominal-state synthesis}.

	\begin{proof}[\textit{Proof} of Theorem~\ref{thm:backward nominal-state synthesis}]
		Letting $t=T$ in \eqref{eq:sol expansion with remainder}, one gets from $x(T)=x^1$ and left multiplication by $V_T(x^1)$ that
		\begin{IEEEeqnarray}{rCl}\label{eq:derivation of the nonlinear input}
			V_T(x^1)(\psi_T(x^1)-x^0)
			&=&-V_T(x^1)\zeta_u(T)Bu\nonumber\\
			&&\hspace{-1cm}+(e^{TV_T(x^1)}-\idty)D\psi_T(x^1)Bu
		\end{IEEEeqnarray}
		by \eqref{eq:exp(A)}. Since $[D\psi_T(x^1)]^{-1}=D\phi_T(\psi_T(x^1))$, left multiplying  \eqref{eq:derivation of the nonlinear input} by $(e^{TV_T(x^1)}-\idty)^{-1}$ and $D\phi_T(\psi_T(x^1))$ yields
		\begin{IEEEeqnarray*}{rCl}\label{eq:derivation of the nonlinear input 1}
			D\phi_T(\psi_T(x^1))\left(e^{TV_T(x^1)}-\idty\right)^{-1}V_T(x^1)(\psi_T(x^1)-x^0)&=&\nonumber\\
			&&\hspace{-8.5cm}Bu- D\phi_T(\psi_T(x^1))\left(e^{TV_T(x^1)}-\idty\right)^{-1}V_T(x^1)\zeta_u(T)Bu.
		\end{IEEEeqnarray*}
		It follows that $u\in\R^k$ solves the implicit equation
		\begin{equation}\label{eq:solving}
			\left[\idty-\cB_T(x^1)\zeta_u(T)\right]Bu=\cB_T(x^1)(\psi_T(x^1)-x^0)
		\end{equation}
		which is~\eqref{eq:backward nominal-state synthesis}, where $\cB_T(x^1)$ is defined by \eqref{eq:matrix B_T(x^1)}. 
	\end{proof}

	Using the same machinery that leads to Theorem~\ref{thm:endpoint error of the forward nominal-state synthesis control}, we obtain the following result. The proof is omitted for brevity.

	\begin{theorem}\label{thm:endpoint error of the backward nominal-state synthesis control}
		Under the assumptions of Theorem~\ref{thm:backward nominal-state synthesis}, let $\overline{x}(\cdot)$ denote the solution of~\eqref{eq:Hopfield nonlinear} corresponding to a constant control $\overline{u} \in \mathbb{R}^k$ satisfying
		\begin{equation}\label{eq:approximation of the backward nominal-state synthesis control}
			B \overline{u} = \left[\idty-e^{-T V_T(x^1)}\right]^{-1} V_T(x^1) \left(\psi_T(x^1) - x^0\right)
		\end{equation}
		where $V_T(x^1) := DN(\psi_T(x^1))$. Then, the following endpoint error estimate holds
		\begin{equation}\label{eq:endpoints error backward nominal-state synthesis control}
			|\overline{x}(T) - x^1| \underset{T \sim 0}{=} \cO(T^2).
		\end{equation}
	\end{theorem}
	
	The following result provides a necessary and sufficient condition for \eqref{eq:approximation of the backward nominal-state synthesis control} to admit at least one solution.
	\begin{corollary}\label{thm:necessary and sufficient cond for solvability nonlinear backward}
		Under the hypotheses of Theorem~\ref{thm:backward nominal-state synthesis}, \eqref{eq:approximation of the backward nominal-state synthesis control} admits at least one solution $\tilde{u} \in \mathbb{R}^k$ if and only if 
		\begin{equation}
			\left[\idty-e^{-T V_T(x^1)}\right]^{-1} V_T(x^1) \left(\psi_T(x^1) - x^0\right)\in\im B.
		\end{equation}
		In this case, the least-norm constant control is given by
		\begin{equation}
			\overline{u} = B^\dagger\,\left[\idty-e^{-T V_T(x^1)}\right]^{-1} V_T(x^1) \left(\psi_T(x^1) - x^0\right).
		\end{equation}
	\end{corollary}
	
	Finally, one has the following result that synthesizes a step function for a large time horizon. 
	
	\begin{theorem}\label{thm:backward nominal-state synthesis step function}
		Let $(x^0, x^1)\in(\R^d)^2$ and $T\ge\tau$, where $\tau>0$ be such that for any $\ell\in\Z$, 
		$i \frac{2\pi \ell}{\tau}\not\in\sigma(DN(\psi_T(x^1)))$. Let $\overline{u}_\tau\in\R^k$ be a solution of 
		\begin{equation}
			B\overline{u}_\tau= \left[\idty-e^{-\tau V_T(x^1)}\right]^{-1} V_T(x^1) \left(\psi_T(x^1) - x^0\right)
		\end{equation}
		where $V_T(x^1)=DN(\psi_T(x^1))$. Then, the solution $x(\cdot)$ to \eqref{eq:Hopfield nonlinear} corresponding to the step function $\overline{u}_{sf}:[0,T]\to\R^k$,
		\begin{equation}
			\overline{u}_{sf}(t) = \begin{cases}
				0&\quad\mbox{if}\quad 0\le t\le T-\tau\\
				\overline{u}&\quad\mbox{if}\quad T-\tau< t\le T
			\end{cases}
		\end{equation}
		satisfies 
		\begin{equation}\label{eq:endpoints error backward nominal-state synthesis step function}
			|x(T)-x^1|\underset{\tau \sim 0}{=} \cO(\tau^2).
		\end{equation}
	\end{theorem}

	\begin{remark}\label{rmk:complementarity of the synthesis}
		A key distinction between the forward and backward nominal-state syntheses lies in their respective spectral conditions: the forward synthesis depends on the spectrum at the initial state~$x^0$, while the backward synthesis relies on that at the target state~$x^1$.
		
		In practice, only one of these conditions may be satisfied, depending on the region of the state space. This asymmetry underscores the \emph{complementarity} of the two syntheses---each offers a valid control strategy under different local spectral properties, making their coexistence practically valuable.
		
		This contrast also reflects a deeper difference between linear and nonlinear control. In the linear case, controllability depends on a uniform spectral condition (cf.~Proposition~\ref{pro:controllability in the linear case}); failure at one point implies failure everywhere. In the nonlinear setting, local spectral conditions at either the initial or target can independently ensure controllability, highlighting a key advantage of the nonlinear synthesis framework. 
	\end{remark}
	
	\subsection{Forward and Backward syntheses}\label{ss:forward and Backward syntheses} 
	
	Although Sections~\ref{sss:Forward nominal-state synthesis} and~\ref{sss:Backward nominal-state synthesis} focus on the \emph{forward} and \emph{backward nominal-state} syntheses, respectively, it is worth emphasizing that both arise from a broader family of synthesis strategies derived from the dual representations of the nonlinear system~\eqref{eq:Hopfield nonlinear} in Theorems~\ref{thm:backward representation} and~\ref{thm:forward representation}. Each representation admits at least two distinct expansion strategies---depending on the chosen integration by parts identity---that lead to meaningful control synthesis formulations.
	
	Specifically, starting from the solution representation \eqref{eq:backward representation}, we observe two synthesis paths:
	\begin{enumerate}
		\item[(i)] Applying the identity $\int_0^T D\psi_t\,dt = \int_0^T t' D\psi_t\,dt$ yields the implicit \emph{backward nominal-state} synthesis
		fully investigated in Section~\ref{sss:Backward nominal-state synthesis}.
		
		\item[(ii)] Using instead $\int_0^T D\psi_t\,dt = \int_0^T (t - T)' D\psi_t\,dt$ gives rise to what we refer to as an implicit \emph{backward initial-state} synthesis. This formulation yields an expression of the control that provides a necessary and sufficient condition for reachability, and it is not analyzed in detail here. 
	\end{enumerate}
	
	Similarly, from the solution representation \eqref{eq:forward representation}, two analogous options arise:
	\begin{enumerate}
		\item[(i)] Using the identity $\int_0^T D\phi_{T-t}\,dt = \int_0^T (t - T)' D\phi_{T-t}\,dt$ leads to the implict \emph{forward nominal-state } synthesis fully developed in Section~\ref{sss:Forward nominal-state synthesis}.
		
		\item[(ii)] Alternatively, using $\int_0^T D\phi_{T-t}\,dt = \int_0^T t' D\phi_{T-t}\,dt$ results in what we refer to as an implicit \emph{forward final-state synthesis}. This yields a formally valid expansion and a necessary condition on the control, and it is not investigated in this work.
	\end{enumerate}
	
	\subsection{Reachability via control synthesis methods}\label{ss:reachability via control synthesis methods}
	
	In this section, we summarize the main results from Sections~\ref{sss:Forward nominal-state synthesis} and~\ref{sss:Backward nominal-state synthesis} into a set of operational insights framed in terms of the input matrix~$B$. Specifically, we analyze how the set of states reachable from a given initial condition $x^0 \in \mathbb{R}^d$ over the interval $[0, T]$ depends on the structure of $B$, under both constant and step-function control strategies synthesized via the proposed methods.
	
	Given $x^0 \in \mathbb{R}^d$, we say that a state $x^1 \in \mathbb{R}^d$ is \emph{c-reachable over} $[0, T]$ from $x^0$ if there exists a constant control $u \in \mathbb{R}^k$ such that the solution $x(\cdot)$ of \eqref{eq:Hopfield nonlinear} corresponding to $u$ satisfies $x(T) = x^1$. The c-reachable set is denoted by (for step-function controls, we denote it by $\mathcal{R}_{sc}(\tau, T, x^0)$)
	\begin{equation}\label{eq:reachable set nonautonomous}
		\mathcal{R}_c(T, x^0) = \left\{ x(T) \, \middle| \, 
		\begin{aligned}
			&x(\cdot) \text{ solves } \eqref{eq:Hopfield nonlinear} \\
			&\text{with constant control } u \in \mathbb{R}^k
		\end{aligned}
		\right\}.
	\end{equation}
	We say that \eqref{eq:Hopfield nonlinear} is \emph{controllable over} $[0, T]$ from $x^0$ with a constant control if $\mathcal{R}_c(T, x^0) = \mathbb{R}^d$. If this holds for every $x^0 \in \mathbb{R}^d$, then~\eqref{eq:Hopfield nonlinear} is said to be \emph{completely controllable} over $[0, T]$ with constant controls.
	
	\subsubsection{The linear case}\label{sss:receability linear case}
	It is convenient to begin with the linear system \eqref{eq:linear model}, which helps understanding the fully nonlinear setting. In the linear control framework, controllability with a time-varying control is guaranteed if the Kalman rank condition or, equivalently, the invertibility of the controllability Gramian holds. These conditions are independent of $T > 0$ and $x^0$, and they allow for flexibility in choosing the input matrix $B \in \mathcal{M}_{d,k}(\mathbb{R})$ satisfying what we refer to as the \emph{time-varying controllability condition}. 
	
	However, as shown in Remark~\ref{rmk:a priori analysis}, the question becomes more delicate when restricting to constant controls. Notably, even the trivial case $B = \idty$ may fail to achieve controllability with a constant control, despite satisfying the time-varying controllability condition.
	
	Under the spectral condition in Proposition~\ref{pro:controllability in the linear case}, Corollary~\ref{cor:necessary and sufficient cond for solvability linear} shows that~\eqref{eq:linear model} is \emph{completely controllable with constant controls} over $[0, T]$ if and only if $k = d$ and $B$ is invertible. Otherwise, if $k < d$ or $B$ is not invertible, then
	\begin{equation}
		\mathcal{R}_c(T, x^0) \varsubsetneq \mathbb{R}^d, \quad \text{for all } x^0 \in \mathbb{R}^d.
	\end{equation}
	Hence, the size of $\mathcal{R}_c(T, x^0)$ is determined by the rank of $B$: the smaller $k$ is relative to $d$, the fewer the number of target states that are reachable from $x^0$ via constant control.
	
	\subsubsection{The Nonlinear Case}
	We now turn to the fully nonlinear setting and analyze the c-reachable set $\mathcal{R}_c(T, x^0)$, leveraging the forward nominal-state synthesis from Section~\ref{sss:Forward nominal-state synthesis}.
	
	Unlike the linear case, the synthesis equation in Theorem~\ref{thm:Forward nominal-state synthesis} is implicit in $u$ unless $\varphi_u(T) = 0$. When $\varphi_u(T) = 0$, the analysis from Section~\ref{sss:receability linear case} directly applies. If $\varphi_u(T) \neq 0$, Proposition~\ref{pro:control synthesis in small time 1} shows that for sufficiently small $T > 0$, the control $u$ satisfies \eqref{eq:equation for u general}.  As a result, \eqref{eq:Hopfield nonlinear} is \emph{completely controllable with constant controls} over $[0, T]$ for small $T > 0$ if and only if $k = d$ and $B$ is invertible. Otherwise, when $k < d$ or $B$ is not invertible, the c-reachable set satisfies
	\[
	\mathcal{R}_c(T, x^0) \varsubsetneq \mathbb{R}^d, \quad \text{for all } x^0 \in \mathbb{R}^d,
	\]
	and its size is again governed by the rank of $B$.
	
	Still, when $\varphi_u(T) \neq 0$, the small-time reachability result allows one to extend the synthesis to arbitrary horizons $T > 0$ via \emph{step-function controls}. Specifically, let $\tau > 0$ be such that for any $\ell\in\Z$, $\frac{2\pi\ell}{\tau}\notin\sigma(DN(\phi_T(x^0))$. Then, for $T \le \tau$, we have $\mathcal{R}_{sc}(\tau, T, x^0) = \mathcal{R}_c(\tau, x^0)$. While for $T > \tau$, Theorem~\ref{thm:forward implicit nominal-state synthesis step function} shows that system \eqref{eq:Hopfield nonlinear} is \emph{completely controllable with step-function controls} over $[0, T]$ if and only if $k = d$ and $B$ is invertible. Otherwise, one has
	\[
	\mathcal{R}_{sc}(\tau, T, x^0) \varsubsetneq \mathbb{R}^d, \quad \text{for all } x^0 \in \mathbb{R}^d
	\]
	and the sc-reachable set structure again depends on $\operatorname{rank}(B)$.

	Let us characterize $\mathcal{R}_c(T, x^0)$ for the input matrix traditionally considered in network neuroscience~\cite{gu2015controllability}, namely
	\begin{equation}\label{eq:example from theoretical neuroscience}
		B = \begin{bmatrix}e_{1} & e_{2} & \dots & e_{k}\end{bmatrix},\qquad k<d,
	\end{equation}
	where  $e_{i}$  is the $i-$th canonical basis vector of $\R^d$.
	
	\begin{proposition}\label{pro:c-recheable set for canonical matrices}
		Let $T > 0$ and $(x^0, x^1)\in(\R^d)^2$. Assume that for any $\ell\in\Z$, 
		$i \frac{2\pi \ell}{T}\not\in\sigma(U_T(x^0))$. If the input matrix $B$ is given by \eqref{eq:example from theoretical neuroscience}, then the c-reachable set satisfies
		\begin{equation}\label{eq:c-recheable set for canonical matrices}
			\begin{split}
				\mathcal{R}_c(T,x^0) \underset{T \sim 0}{=} \phi_T(x^0) &+ \ker M_T(x^0) + \cO(T^2),\\
				\dim(\ker M_T(x^0)) &= k,
			\end{split}
		\end{equation}  
		and the constant input $u\in\R^k$ that steers~\eqref{eq:Hopfield nonlinear} from $x^0\in\R^d$ to $x^1\in\phi_T(x^0) + \ker M_T(x^0)$ is given by
		\begin{equation}\label{eq:control set for canonical matrices}
			u = V_T(x^0)\left(x^1-\phi_T(x^0)\right).
		\end{equation}
		Here $U_T(x^0):=DN(\phi_T(x^0))$, $V_T(x^0)\in\cM_{k,d}(\R)$ collects the $k$ nonzero rows of $\cB_T(x^0):=[e^{TU_T(x^0)}-\idty]^{-1}U_T(x^0)$, and $M_T(x^0)\in\cM_{d-k,d}(\R)$ collects its $d-k$ zero rows.
	\end{proposition}
	\begin{proof}
		First, it follows from Theorem~\ref{thm:endpoint error of the forward nominal-state synthesis control} that the solution $x(\cdot)$ to~\eqref{eq:Hopfield nonlinear}
		corresponding to a constant control $u\in\R^k$ solving
		\begin{equation}\label{eq:deriv 1}
			Bu = \left[e^{T U_T(x^0)} - \idty\right]^{-1} U_T(x^0) \left(x^1 - \phi_T(x^0)\right)
		\end{equation}
		satisfies (see~\eqref{eq:sol expansion with remainder 1 tilde u})
		\begin{equation}\label{eq:deriv 2}
			x(T)\underset{T \sim 0}{=}x^1+\cO(T^2).
		\end{equation}
		Next, with $B$ as in \eqref{eq:example from theoretical neuroscience}, one deduces from \eqref{eq:deriv 1} that
		\begin{equation}\label{eq:deriv 3}
			\begin{bmatrix}
				\idty_k\\
				O_{d-k,k}
			\end{bmatrix}u = \begin{bmatrix}
				V_T(x^0)\\
				M_T(x^0)
			\end{bmatrix}\left(x^1-\phi_T(x^0)\right)
		\end{equation}
		where $\idty_k\in\cM_k(\R)$ is the identity matrix, and $O_{d-k,k}\in\cM_{d-k,k}(\R)$ is the zero matrix. Since $\cB_T(x^0)$ is invertible, one has $\operatorname{rank}M_T(x^0)=d-k$ and therefore $\dim(\ker M_T(x^0))=d-(d-k)=k$. Finally, \eqref{eq:deriv 3} is equivalent to~\eqref{eq:control set for canonical matrices} and
		\begin{equation}\label{eq:deriv 4}
			x^1\in \phi_T(x^0)+\ker M_T(x^0).
		\end{equation}
		Combining~\eqref{eq:reachable set nonautonomous}, \eqref{eq:deriv 2} and~\eqref{eq:deriv 4} yields~\eqref{eq:c-recheable set for canonical matrices}.
	\end{proof}
	
	\begin{remark}\label{rmk:nonlinear underactuated approximate reachable set}
		In Proposition~\ref{pro:c-recheable set for canonical matrices}, the first-order characterization of the c-reachable set as an affine subspace is valid as long as the desired target shift $x^1 - \phi_T(x^0)=\cO(T)$. This ensures that the required constant input \eqref{eq:control set for canonical matrices} remains bounded, since to leading order $T\to 0$, one has $|u| \approx |x^1 - \phi_T(x^0)|/T$.
		
		Using the thin QR factorization of $M_T(x^0)^\top\in\R^{d\times(d-k)}$,
		\[
		M_T(x^0)^\top = \begin{bmatrix}
			Q_1&Q_2
		\end{bmatrix}\begin{bmatrix}
			R_1\\0
		\end{bmatrix},\;\begin{cases}
			Q_1\in\R^{d\times(d-k)},\,Q_2\in\R^{d\times k},\\R_1\in\R^{(d-k)\times(d-k)},
		\end{cases} 
		\]
		one finds that $\ker M_T(x^0)=\operatorname{span}(Q_2)$.
	\end{remark}
	
	\begin{remark}\label{rmk:nonlinear underactuated}
		First, the c-reachable set characterization and the associated constant input synthesis in Proposition~\ref{pro:c-recheable set for canonical matrices} for a small time horizon $T>0$ when the input matrix $B$ is given by~\eqref{eq:example from theoretical neuroscience} can be improved by explicitly considering the second derivative $D^2\phi_{T-t}$ in the expansion. We defer this analysis to future works. Next, let $T > 0$ and $x^0 \in \mathbb{R}^d$ be such that $DN(\phi_T(x^0))$ satisfies the spectral condition in Theorem~\ref{thm:Forward nominal-state synthesis}. Future work should investigate the structure and effective dimension of the c-reachable sets $\mathcal{R}_c(T, x^0)$ as functions of $\operatorname{rank}(B)$ for a general input matrix $B\in\cM_{k,d}(\R)$.
	\end{remark}

	\section{Comments on the main results of Section~\ref{s:exact controllability}}\label{s:general comments}
	
	First, if the spectral norm condition $\|W\| < \lambda_{\min}(D)$ is satisfied, then the eigenvalue-related assumptions required in Proposition~\ref{pro:controllability in the linear case},Theorems~\ref{thm:Forward nominal-state synthesis}, and~\ref{thm:backward nominal-state synthesis} are automatically fulfilled, as discussed, e.g., in Remark~\ref{rmk:linear case}. 
	Note that this condition ensures that the nonlinear system \eqref{eq:Hopfield nonlinear} is contracting. 
	
	Contracting dynamics possess a range of desirable properties--see \cite{lohmiller1998contraction}--and contraction has become a standard structural assumption in many recent works on recurrent neural networks, such as \cite{davydov2022non,centorrino2023euclidean}. While the condition on $\|W\|$ is a sufficient criterion for contraction, it serves here as a simple and practically verifiable condition that ensures well-posedness of the synthesis proposed in this work.
	
	More generally, the control syntheses developed in this paper apply to any nonlinear system of the form $\dot{x}(t)=N(x(t))+ Bu$, where $u\in\R^k$ is constant and $N\in C^2(\R^d;\R^d)$ satisfies the regularity conditions
	\begin{equation}\label{eq:general assumption on DN and D2N}
		\sup_{x \in \R^d} \|DN(x)\| \le C, \qquad \sup_{x \in \R^d} \|D^2N(x)\| \le C
	\end{equation}
	for some constant $C > 0$. Notably, we do not require any explicit boundedness condition on the vector field $N$ itself, making the framework broadly applicable.
	
	This flexibility is particularly relevant for applications in machine learning, where recurrent neural networks (RNNs) often employ unbounded and non-smooth activation functions such as the ReLU: $f_i(x_i) = \mathrm{ReLU}(x_i) = \max(0, x_i) $. ReLU is differentiable everywhere except at $x_i = 0$, and it is globally Lipschitz on $\R$. In this case, the solution representations \eqref{eq:backward representation} and \eqref{eq:forward representation} no longer apply, and our control syntheses cannot be directly used. To overcome this limitation, one can consider smooth approximations of the ReLU function (e.g., softplus \cite{glorot2011deep},  swish function \cite{ramachandran2017searching}), which restore differentiability and allow the application of our synthesis results.

	\section{Examples and numerical simulations}\label{s:examples and numerical simulation} 
	This section presents some numerical simulations that underpin our theoretical study. The pipeline of the experiments and analysis is visualized in Fig.~\ref{fig:pipeline}. We included three major types of models: linear systems, ``vanilla'' tanh RNNs, and MINDy models \cite{singh2020estimation}. While vanilla RNNs are widely used in theoretical neuroscience as task-performing models, MINDy models represent another branch of models that directly approximate experimental neural data, demonstrating the different applicative contexts of our theoretical framework.
	
	\begin{figure*}[t!]
		\centering
		\includegraphics[width=0.95\linewidth, trim={0 6.8cm 0 0}, clip=true]{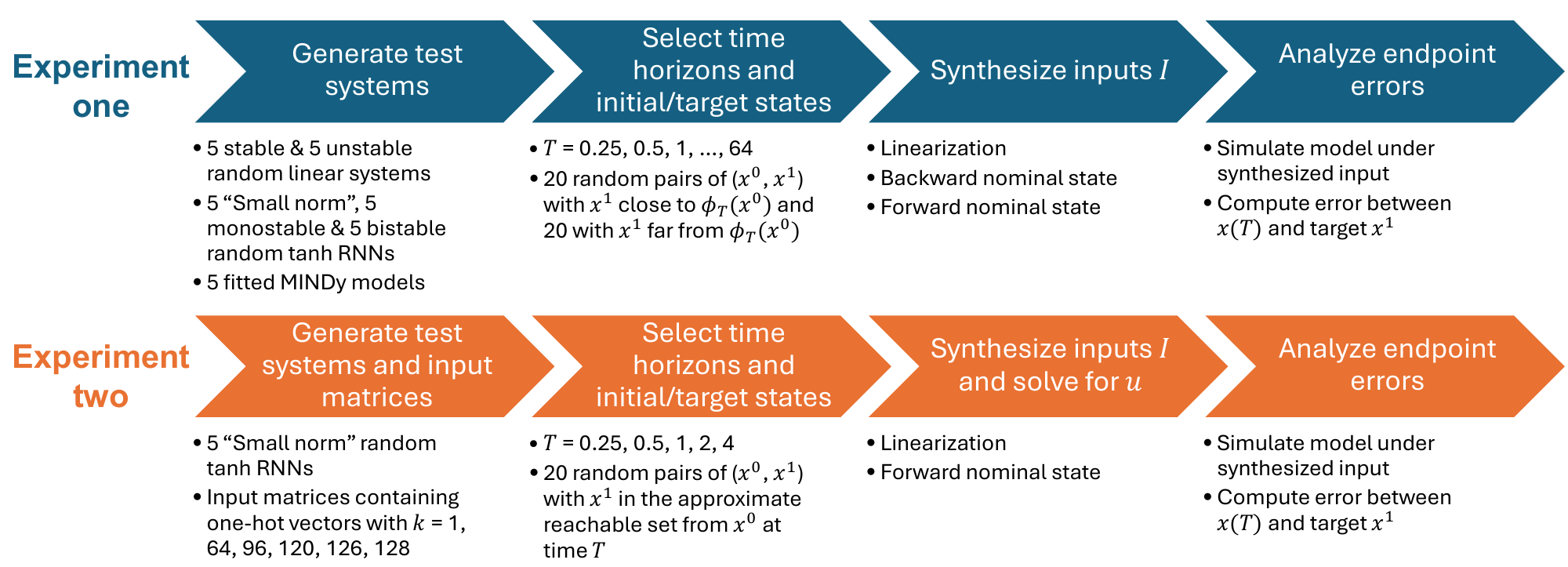}
		\caption{ \textbf{Pipeline of the numerical experiments.}}
		\label{fig:pipeline}
	\end{figure*}
	
	\subsection{Implementation}
	
	We implement the linear system control synthesis \eqref{eq:linear control particular}, the nonlinear forward nominal-state synthesis \eqref{eq:approximation of the forward nominal-state synthesis control}, and the nonlinear backward nominal-state synthesis \eqref{eq:approximation of the backward nominal-state synthesis control} in \verb|python|. The scripts will be available after publication. For the nonlinear synthesis, the flow is numerically integrated using the \verb|odeint| function from the \verb|torchdiffeq| package \cite{chen2018neuralode}. We adopted the Runge-Kutta method of order 7(8) of Dormand-Prince-Shampine, which is the highest-order method available from the package. The method accepts a new step if $\text{RMS}(\varepsilon) < 10^{-14} + 10^{-13} * \text{RMS}(x)$, where $\text{RMS}(\cdot)$ represents the root mean square norm, $\varepsilon$ being the estimated error and $x$ the current state. We found that this setting provided reasonable control of the numerical error related to flow integration. The Jacobian matrices were computed through auto-differentiation in \verb|pytorch|. Experiments were conducted on a desktop computer with an Nvidia RTX 3080 GPU. The run time for each synthesis increased as the time horizon $T$ and number of features (neurons) $d$ increased, but remained within the scope of several hundred milliseconds to several seconds. 
	For a fixed $T$ and multiple desired trajectories (i.e., different pairs of $(x^{0}, x^{1})$), the synthesis is efficiently parallelized, and the run time remains almost constant when increasing the number of trajectories from one to $1000$.
	
	To construct a $d$-dimensional linear system, we randomly sampled the entries of $W$ from the normal distribution $\mathcal{N}(0, 1 / d)$. The decay matrix was set to $(\lambda_{\max}(W) - \lambda_{0})\idty$, where $\lambda_{\max}(W)\in\mathbb{R}$ is the maximum of the real parts of the eigenvalues of $W$ and $\lambda_{0}\in\mathbb{R}$ is the desired maximum real part of the eigenvalues of the dynamic matrix $A = W - D$. We tried both $\lambda_{0} = -0.1$ and $\lambda_{0} = 0.1$ to construct stable and unstable systems respectively.
	
	To construct a ``vanilla'' RNN, we set $D = \idty$ and the activation function to $\tanh$. We followed \cite{schuessler2020dynamics} and imposed a ``random plus low rank'' structure on the connectivity matrix $W = J + \mathbf{mn}^{\top}$. The matrix $J \in \cM_d(\R)$ was sampled from the normal distribution $\mathcal{N}(0, g^{2} / d)$, where $g$ is a scaling factor. $m, n\in\cM_{d, p}(\mathbb{R})$ determines the rank of the low-rank component. Such a type of model appears frequently in theoretical neuroscience studies \cite{mastrogiuseppe2018linking}. We considered three subtypes of vanilla RNNs. The first subtype is referred to as `small norm tanh RNN', constructed with $g = 0.5$ and $p = 0$ (no low-rank component). We found that such RNNs satisfied $\|WDf(0)\|\approx 1 = \lambda_{\min}(D)$. For the next two subtypes, $g$ was set to 0.9 \cite{schuessler2020dynamics} and the spectral norm condition was violated. $\mathbf{m} \in \mathbb{R}^{d}$ was sampled from standard normal distribution. $\mathbf{n} \in \mathbb{R}^{d}$ was sampled from normal distribution $\mathcal{N}(0, 1 / d^2)$ for the second subtype and set as $\mathbf{n} = \frac{1.1}{d} \mathbf{m}$ for the third subtype. It can be shown that systems of the second subtype will be monostable while systems of the third subtype will be bistable \cite{schuessler2020dynamics} as the number of neurons $d$ tends to infinity.
	
	To further evaluate the method on realistic models fit to \emph{experimental data}, we also included a set of Mesoscale Individualized Neurodynamic (MINDy) models from \cite{singh2020estimation}. MINDy models contained 100 interconnected units representing 100 brain areas, and the parameters were optimized to approximate the activation time series of these areas measured through functional magnetic resonance imaging (fMRI). Unlike most RNNs, the activation function of MINDy is heterogeneous: $f(s) = \sqrt{\alpha^{2} + (bs + 0.5)^{2}} - \sqrt{\alpha^{2} + (bs - 0.5)^{2}}$, where $b = 20/3$ is fixed and $\alpha$ is optimized over the data and differs across the 100 units. It is also worth noting that the origin is unstable in most of the MINDy models, indicating that the spectral norm condition $\|WDf(0)\| < \lambda_{\min}(D)$ was not met. However, we found that the models satisfy the eigenvalue condition, which enabled the synthesis.
	
	In the first experiment, we analyzed the endpoint error of the controlled trajectory for $T \in \{2^{-2}, 2^{-1}, \dots, 2^{6}\}$ when $B = \idty$. For each $T$, we randomly generated 5 stable and 5 unstable linear systems, 5 small-norm, 5 monostable and 5 bistable tanh RNNs, and randomly selected 5 fitted MINDy models (from a pool of 106 models). All models were 100-dimensional.  Then, we randomly generated 40 initial states $x^{0} \sim \mathcal{N}(0, \idty)$. To investigate how the deviation of $x^{1}$ from the autonomous flow end point $\phi_{T}(x^{0})$ influences the performance of the method, we set $x^{1} = \phi_{T}(x^{0}) + \varepsilon$ where $\varepsilon\sim \mathcal{N}(0, \sigma^{2}\cdot\idty)$. In half of the trials, $\sigma^{2}$ was set to 0.1 (``small deviation''); in the other half, $\sigma^{2}$ was set to 0.5 (``large deviation''). Control input was computed using the nonlinear forward \eqref{eq:approximation of the forward nominal-state synthesis control} and backward \eqref{eq:approximation of the backward nominal-state synthesis control} syntheses, and the linear synthesis \eqref{eq:linear control particular}. We recall that for the nonlinear system~\eqref{eq:Hopfield nonlinear}, when $B=\idty$, the linearized system at $x^{0}$ is given by
	\begin{equation}
		\dot{x}(t) = DN(x^{0})x(t) + N(x^{0}) + u.
	\end{equation}
	Using \eqref{eq:linear control particular}, the following input $u$ drives the linearized system from $x^0$ to $x^1$ in exact time $T$
	\begin{equation}\label{eq: linearized at x0}
		u = \left(e^{T A}-\idty\right)^{-1}A\left(x^1 - e^{T A}x^0\right) - N(x^{0})
	\end{equation}
	where $A = DN(x^{0})$. We have also tried to linearize the systems at the origin (which is always a fixed point) and obtained qualitatively similar results.
	
	
	In the second experiment, we followed the tradition in network neuroscience \cite{gu2015controllability} and considered the input matrix defined by \eqref{eq:example from theoretical neuroscience}. For simplicity, we focused on 128-dimensional ``small norm'' tanh RNNs (with $g = 0.5$ and $p = 0$) and $T \in \{0.25, 0.5, 1, 2, 4\}$. We considered $k \in \{1, 64, 96, 120, 126, 128\}$, ranging from an extremely underactuated system to a fully actuated system. We randomly generated 5 RNNs and selected 20 initial states $x^{0}\sim\mathcal{N}(0, \idty)$ for each model and each combination of $T$ and $k$. Here, $x^{1} = \phi_{T}(x^{0}) + Q_{2}\xi$ by Proposition~\ref{pro:c-recheable set for canonical matrices}, and $Q_2$ following Remark~\ref{rmk:nonlinear underactuated approximate reachable set}, with entries of $\xi$ drawn from $\mathcal{N}(0, 0.01)$. 
	
	\subsection{Results}
	
	\begin{figure}[t!]
		\centering
		\includegraphics[width=\linewidth]{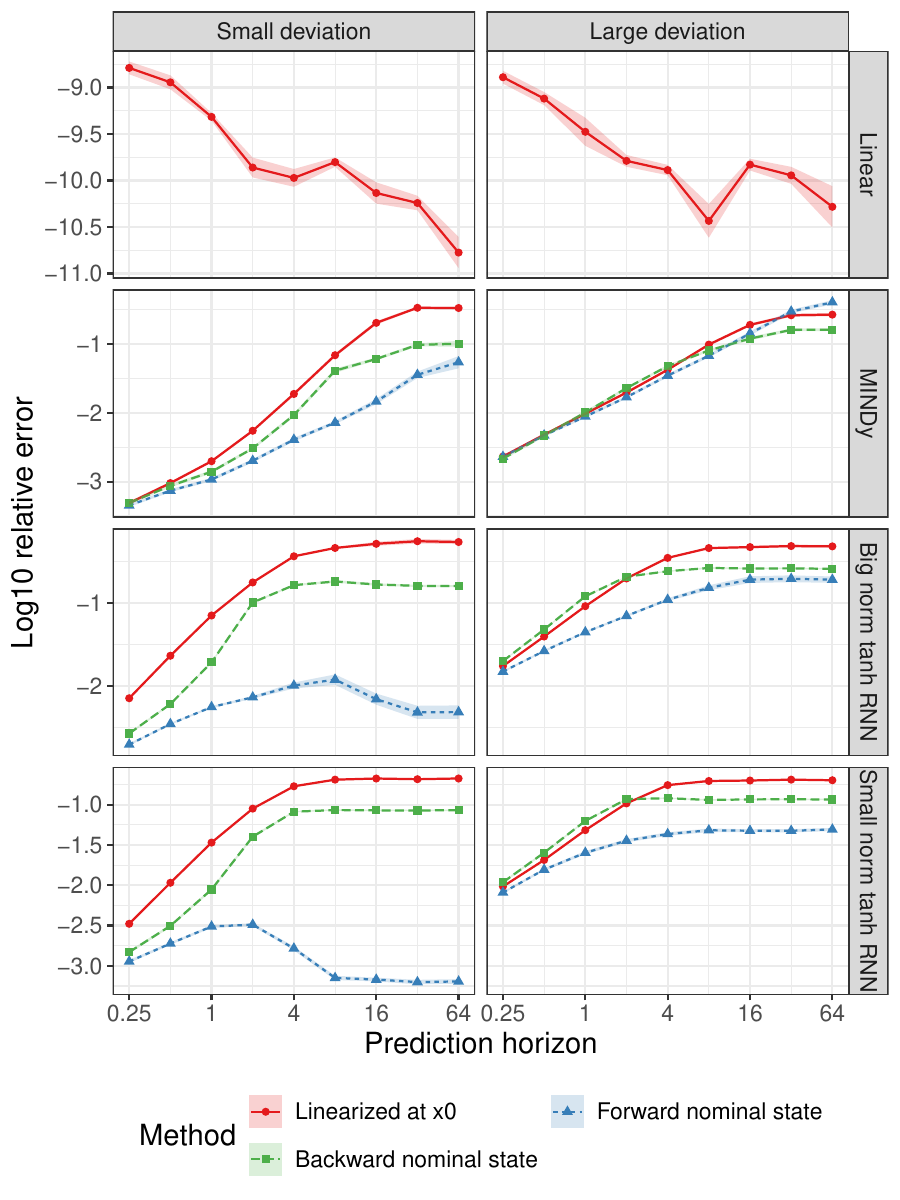}
		\caption{\textbf{Relative endpoint error under synthesized input.} The $x$-axis (log scale) indicates the time horizon $T$. The $y$-axis represents the common logarithm of the ratio between the Euclidean norm of the endpoint error $x(T) - x^{1}$ and the Euclidean norm of $x^{1} - x^{0}$.  Results were organized according to the combination of model type (rows) and how $x^1$ were specified (columns). Monostable and bistable RNNs were merged into a single category of ``Big norm tanh RNN'' due to similarity. Results using linearization \eqref{eq: linearized at x0}, nonlinear forward nominal-state synthesis \eqref{eq:approximation of the forward nominal-state synthesis control}, and nonlinear backward nominal-state synthesis \eqref{eq:approximation of the backward nominal-state synthesis control} were shown in red, blue and green respectively. Line plots and error bands represent the mean and 95\% confidence interval of log relative error.}
		\label{fig:error}
	\end{figure}
	
	Results of the first experiment were summarized in Fig.~\ref{fig:error}. Monostable and bistable tanh RNNs were combined into ``Big norm tanh RNN'' due to similarity. For linear systems, the error remained small. For nonlinear systems, error generally increased as the horizon $T$ increased. The forward nominal state synthesis performed better than the backward nominal state synthesis, and both of them generally performed better than linearization, particularly for large $T$. The differences were more evident when $x^1 - \phi_{T}(x^0)$ is small (with standard deviation 0.1, left panels), where the forward synthesis sometimes outperformed linearization by orders of magnitude. Note that the MINDy models were trained to approximate timeseries with unit variance, so a ``small deviation'' with standard deviation 0.1 already represents a physically sizable difference, indicating the practical significance of our method.
	
	\begin{figure}[t!]
		\centering
		\includegraphics[width=\linewidth, trim={0 0.95cm 0 0}, clip=true]{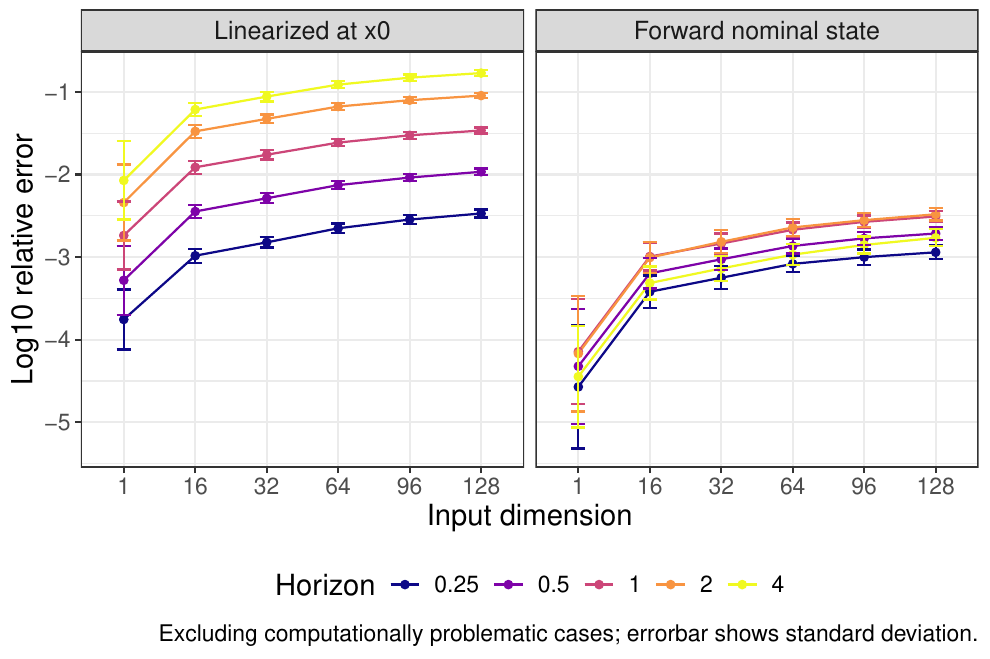}
		\caption{ \textbf{Relative endpoint error under synthesized input for underactuated systems.} We considered ``small norm tanh RNNs'' of 128 dimension. The $x$-axis (ordinal scale) indicates the input dimension $k$. The $y$-axis represents the common logarithm of the relative endpoint error. Results using linearized and nonlinear forward synthesis methods were separated into the two panels. Color indicates different time horizons $T$. Line plots and error bars represent the mean and standard deviation of log relative error over 20 experiments.}
		\label{fig:error-nonId}
	\end{figure}
	
	Results of the second experiment were summarized in Fig.~\ref{fig:error-nonId}. The nonlinear synthesis was still much better than linearization in underactuated systems, as expected. However, surprisingly, the endpoint error did not change dramatically and even decreased as the number of actuators decreased, as long as $x^1$ is close to the reachable set. These findings suggest that Proposition~\ref{pro:c-recheable set for canonical matrices} did provide a good estimation of the reachable set, and that the synthesis worked in underactuated systems just as good as (if not better than) fully-actuated systems as long as $x^{1}$ is reachable.
	
	\section{Concluding remarks and perspectives}\label{s:concluding remarks and discussion}
	This paper addressed the control synthesis problem for a class of nonlinear Hopfield-type recurrent neural networks motivated by neurostimulation applications. Using a solution representation that generalizes the variation of constants formula, we derived constant and piecewise constant inputs capable of steering the network to a desired target within a prescribed time interval. For linear activation functions, the exact controllability with constant input is possible over arbitrary time horizons. For a small time horizon, we also provided a characterization of the reachable set when the input matrix $B$ directly actuates a subset of nodes, showing that its dimension equals the rank of $B$, and that its basis can be computed efficiently via a thin QR factorization. Moreover, the constant input that guarantees reachability is given by a closed-form algebraic condition, which makes the synthesis directly applicable for tDCS. 
	
	Future work will explore the synthesis of time-varying control inputs based on the proposed framework, with an emphasis on robustness to parameter uncertainty and external disturbances. As large-scale systems like~\eqref{eq:Hopfield nonlinear} are highly sensitive to such perturbations, designing controls that adapt dynamically while minimizing energy costs remains a key challenge---one that links controllability with energetic performance.
	
	\appendices
	
	\section{General results on the drift and its associated flow maps}\label{s:general results on N}

	This section contains various supplementary results related to the vector field $N$ that were used intensively in the previous sections. The first result concerns some useful properties of the  Fréchet-differential $DN$ and the second-differential $D^2N$. 
	\begin{lemma}\label{lem:properties on N}
		The vector field $N$ defined in \eqref{eq:nonlinear vector field N} belongs to $C^2(\R^d;\R^d)$, and it is Lipschitz continuous from $\R^d$ into itself. Moreover, the following hold
		\begin{equation}\label{eq:differential of N}
			\begin{aligned}
				DN(x) &= -D+WDf(x)\\
				\|DN(x)\|&\le \lambda_{\max}(D)+\|W\|
			\end{aligned}\qquad\forall x\in\R^d.
		\end{equation}
		\begin{equation}\label{eq:differential second of N}
			\begin{aligned}
				D^2N(x) &= WD^2f(x)\\
				\|D^2N(x)\|&\le \|f''\|_\infty\|W\|
			\end{aligned}\qquad\forall x\in\R^d.
		\end{equation}
		Here $\|f''\|_\infty=\max_{i\in\llbracket 1,d\rrbracket}\|f_i''\|_\infty$, $Df(x) = (f_i'(x_i)\delta_{i,j})$ and $D^2f(x)$ is a third order tensor such that $D^2f(x)yz\in\R^d$ for all $y,\, z\in\R^d$, and the $i$-th element is given, by 
		\begin{equation}\label{eq:remark on the second derivative of f}
			[D^2f(x)yz]_{|_i}=\sum_{j,k=1}^{d}\frac{\partial^2f_i}{\partial x_j\partial x_k}(x)y_jz_k=f_i''(x_i)y_iz_i.
		\end{equation}
		
	\end{lemma}
	\begin{proof}
		Let $x,y\in\R^d$. Since $\triplenorm{Df(x)}\le 1$, one has
		\begin{equation}
			|f(x)-f(y)|\le\|Df(x)\| |x-y|\le|x-y|.
		\end{equation}
		One deduces that
		\begin{equation}
			|N(x)-N(y)|\le (\lambda_{\max}(D)+\|W\|)|x-y|
		\end{equation}
		showing that $N$ is globally Lipschitz continuous from $\R^d$ into itself.  On one hand, one has obviously that $N\in C^2(\R^d; \R^d)$ since $f\in C^2(\R^d; \R^d)$ by assumption. It follows that $DN(x)h = -Dh+WDf(x)h$ is linear w.r.t. $h\in\R^d$ and satisfies
		\[
		|DN(x)h|\le (\lambda_{\max}(D)+\|W\|)|h|
		\]
		so that \eqref{eq:differential of N} is satisfied. On the other hand, for all $y,\, z\in\R^d$, one has
		\[
		D^2N(x)yz = WD^2f(x)yz.
		\]
		Since $f:\R^d\to\R^d$, $D^2f(x)$ is a third order tensor, one has $D^2f(x)yz\in\R^d$ for all $y,\, z\in\R^d$. Therefore, \eqref{eq:remark on the second derivative of f} is satisfied because of the ``diagonal'' structure of $f$, the second partial derivatives of $f_i$ are non-zero only when $j=k=i$. Therefore,
		\[
		|D^2f(x)yz|\le \|f''\|_\infty\|W\||y||z|
		\]
		by Cauchy-Schwarz inequality.
	\end{proof} 
	
	The next result concerns the differential of the flow $\phi_t$ of the vector field $N$.  One has the following general estimates.
	\begin{lemma}\label{lem:general estimates flows}
			Let $\beta\in C^0(\R_+,\R^d)$ and $\gamma\in C^1(\R_+,\R)$ be such that $\gamma(0)\ge 0$. Then, for every $t\ge 0$, it holds
			\begin{equation}\label{eq:spectral norm of D_Phi complicated}
				\|D\phi_{\gamma(t)}(\beta(t))\|\le e^{-\gamma(0)\Gamma}e^{\Lambda\int_0^t|\dot{\gamma}(s)|\,ds},
			\end{equation}
			and
			\begin{IEEEeqnarray}{rCl}\label{eq:spectral norm of D^2_Phi}
				\|D^2\phi_{\gamma(t)}(\beta(t))\|&\le&\Lambda_1e^{\Lambda \int_0^t|\dot{\gamma}(\tau)|d\tau}\left[\gamma(0)e^{2\Lambda\gamma(0)}\right.\nonumber\\
				&&\left.+\|\dot{\gamma}\|_\infty te^{-2\gamma(0)\Gamma}e^{\Lambda \int_0^t|\dot{\gamma}(\tau)|d\tau}\right].
			\end{IEEEeqnarray}
	\end{lemma} 
	\begin{proof}
			Let $y\in\R^d$ and $\gamma\in C^1(\R_+,\R)$ be such that $\gamma(0)\ge 0$. By differentiating \eqref{eq::nonlinear flow} with respect to $y$, we find 
			\begin{equation*}
				D\phi_{\gamma(t)}(y)=D\phi_{\gamma(0)}(y)+\int_0^t\dot{\gamma}(s)DN(\phi_{\gamma(s)}(y))D\phi_{\gamma(s)}(y)\,ds.
			\end{equation*}
			It follows that
			\begin{IEEEeqnarray}{rCl}\label{eq:with gamma}
				\|D\phi_{\gamma(t)}(y)\|&\le&\|D\phi_{\gamma(0)}(y)\|+\Lambda\int_0^t|\dot{\gamma}(s)|\|D\phi_{\gamma(s)}(y)\|\,ds\nonumber\\
				&\le&e^{\Lambda \int_0^t|\dot{\gamma}(s)|\,ds}e^{-\gamma(0)\Gamma}
			\end{IEEEeqnarray}
			by \eqref{eq:differential-Phi enhance} and Gronwall's lemma. Let $\beta\in C^0(\R_+,\R^d)$, then for a fixed $t\ge 0$, the map $s\mapsto D\phi_{\gamma(t)}(\beta(s))$ is continuous by composition, and because of \eqref{eq:with gamma}, we find
			\begin{equation}
				\|D\phi_{\gamma(t)}(\beta(s))\|\le e^{\Lambda \int_0^t|\dot{\gamma}(s)|\,ds}e^{-\gamma(0)\Gamma},
			\end{equation}
			which proves \eqref{eq:spectral norm of D_Phi complicated} by continuity and taking the $\lim\sup$ as $s\to t$. By differentiating \eqref{eq::nonlinear flow} twice with respect to $y$, we find
			\begin{IEEEeqnarray*}{rCl}
				D^2\phi_{\gamma(t)}(y)&=&\int_0^t\dot{\gamma}(s) D^2N(\phi_{\gamma(s)}(y))D\phi_{\gamma(s)}(y)D\phi_{\gamma(s)}(y)\,ds\nonumber\\
				&&\hspace{-0.9cm}+D^2\phi_{\gamma(0)}(y)+\int_0^t\dot{\gamma}(s)DN(\phi_{\gamma(s)}(y))D^2\phi_{\gamma(s)}(y)\,ds,
			\end{IEEEeqnarray*}
			which yields
			\begin{IEEEeqnarray}{rCl}\label{eq:norm D^2_alpha_phi on R}
				\|D^2\phi_{\gamma(t)}(y)\|&\le&\Lambda _1e^{-2\gamma(0)\Gamma}e^{2\Lambda \int_0^t|\dot{\gamma}(\tau)|d\tau}\int_0^t|\dot{\gamma}(\tau)|\,d\tau\nonumber\\
				&&+e^{\Lambda \int_0^t|\dot{\gamma}(\tau)|d\tau}\|D^2\phi_{\gamma(0)}(y)\|
			\end{IEEEeqnarray}
			by \eqref{eq:differential second of N}, \eqref{eq:with gamma} and Gronwall's lemma. Applying \eqref{eq:norm D^2_alpha_phi on R} with $D^2\phi_t(y)$, i.e., $\gamma(t)=t$, $\gamma(0)=0$, and $D^2\phi_0(y)=0$ yields
			\begin{equation}
				\|D^2\phi_{\gamma(0)}(y)\|\le \Lambda _1\gamma(0)e^{2\Lambda\gamma(0)}.
			\end{equation}
			It follows that
			\begin{IEEEeqnarray}{rCl}
				\|D^2\phi_{\gamma(t)}(y)\|&\le&\Lambda_1e^{\Lambda \int_0^t|\dot{\gamma}(\tau)|d\tau}\left[\gamma(0)e^{2\Lambda\gamma(0)}\right.\nonumber\\
				&&\left.+\|\dot{\gamma}\|_\infty te^{-2\gamma(0)\Gamma}e^{\Lambda \int_0^t|\dot{\gamma}(\tau)|d\tau}\right].
			\end{IEEEeqnarray}
			One completes the proof of \eqref{eq:spectral norm of D^2_Phi} by arguing as previously. 
	\end{proof}
	
	It should be noted that when $\gamma(t)=t$, $\|D\phi_t(x)\|$ does not systematically grow with $t\ge 0$. One proves the following by arguing as in \eqref{eq:dissip proof} below.
	\begin{lemma}\label{lem:differential-Phi enhance} 
		For every $t\ge 0$, and all $x\in\R^d$, it holds
		\begin{equation}\label{eq:differential-Phi enhance}
			\begin{split}
				\|D\phi_t(x)\|&\le e^{-(\lambda_{\min}(D)-\|W\|)t}\\
				\|D\psi_t(x)\|&\ge e^{(\lambda_{\min}(D)-\|W\|)t}.
			\end{split}
		\end{equation}
	\end{lemma}

	\section{Proof of some of the results from Section~\ref{s:general notations}}\label{s:Supplementary results for representation of solutions}
	
	\subsection{Proof of Theorem~\ref{thm:backward representation}}\label{ss:proof of representation of solution}
	\begin{proof}  Under Assumptions~\ref{ass:general assumption}, \eqref{eq:Hopfield nonlinear} has a unique solution $x\in C^3([0, T];\R^d)$ by the Cauchy-Lipschitz theory (see, for instance, \cite[Chapter 2]{bressan2007introduction}). Let us prove that $x$ can be represented by \eqref{eq:backward representation}.
		Let $y:t\in[0, T]\mapsto y(t)\in\R^d$ be such that $x(t) = \phi_t(y(t))$. Then, $y(t)=\psi_t(x(t))$, and $y$ is derivable w.r.t. $t$. Taking the derivative of $x(t) = \phi_t(y(t))$, we find the following by using the chain rule formula, \eqref{eq:Hopfield nonlinear}, and \eqref{eq::nonlinear flow}
		\begin{equation}
			N(x(t))+Bu=\dot{x}(t)=N(x(t))+D\phi_t(y(t))\dot{y}(t).
		\end{equation}
		Using \eqref{inverse of the differential of the flow} and $x(t) = \phi_t(y(t))$, we find $ [D\phi_t(y(t))]^{-1}=D\psi_t(x(t))$.
		It follows that $y$ solves
		\begin{equation}\label{eq:equation satisfying y}
			\dot{y}(t)=D\psi_t(x(t))Bu,\qquad y(0)=x^0.
		\end{equation}
		Integrating \eqref{eq:equation satisfying y} over $[0, t]$ yields \eqref{eq:backward representation}. 
		
		Conversely, let us show that $x$ given by \eqref{eq:backward representation} solves \eqref{eq:Hopfield nonlinear} and belongs to $C^3([0, T];\R^d)$. First, $x(0)=\phi_0(x^0)=x^0$ and $x\in C^0([0, T]; \R^d)$ by composition.  Otherwise, there exists $(t_n)\subset [0, T]$, $t_{*}\in[0, T]$ with $t_n\to t_{*}$ and $\varepsilon>0$ such that $|\psi_{t}(x(t_n))-\psi_{t_{*}}(x(t_{*}))|\ge\varepsilon$. In fact, $x\in C^0([0, T]; \R^d)$ if and only if $t\mapsto\psi_{t}(x(t))$ belongs to $C^0([0, T]; \R^d)$ since $\psi_{t}$ is invertible and $C^3$ w.r.t. $t\in\R$. We have from \eqref{eq:backward representation} that
		\[
		\psi_{t_n}(x(t_n))-\psi_{t_{*}}(x(t_{*})) = \int_{t_*}^{t_n}D\psi_{s}(x(s))Bu\,ds
		\]
		which, using \eqref{eq:spectral norm of D_Phi complicated} with $\gamma(s)=-s$ and $\beta(s)=x(s)$, implies
		\begin{equation}
			|\psi_{t_n}(x(t_n))-\psi_{t_{*}}(x(t_{*}))|\le |Bu| e^{\Lambda_1(T-t_*)}|t_n-t_*|.
		\end{equation}
		It follows that $|\psi_{t_n}(x(t_n))-\psi_{t_{*}}(x(t_{*}))|\to 0$ as $n\to\infty$, which is inconsistent. Invoking Lebesgue derivation theorem under the integral sign, the r.h.s. of \eqref{eq:backward representation} belongs to $C^2([0, T]; \R^d)$ by composition and an iterative argument since\footnote{See, for instance, \cite[Chapter~15, Theorem~2]{hirsch1974differential}.} $\phi_t(\cdot)\in C^2(\R^d)$. Let
		\begin{equation}
			z(t)= x^0+\int_{0}^{t}D\psi_s(x(s))Bu\,ds
		\end{equation}
		so that $x(t) = \phi_t(z(t))$. Deriving \eqref{eq:backward representation} w.r.t. $t$ yields
		\begin{equation}\label{eq:derivation x}
			\dot{x}(t) = N(\phi_t(z(t)))+Bu=N(x(t))+Bu
		\end{equation}
		by $D\phi_t(z(t))D\psi_t(x(t))=\idty$. Since the r.h.s. of \eqref{eq:derivation x} defines a continuous function in $t$, one deduces that $x$ given by \eqref{eq:backward representation} belongs to $C^3([0, T]; \R^d)$ and solves \eqref{eq:Hopfield nonlinear}. 
	\end{proof}
	
	\section{Proof of some of the results from Section~\ref{ss:controllability nonlinear activation function}}\label{s:Supplementary results for Forward nominal-state synthesis}
	
	We start with the following informative results.
	
		\begin{lemma}\label{lem:key}
			Let $(t, y)\in\R_+\times\R^d$, $\phi_t(y)=DN(\phi_t(y))$ and $V_t(y)=DN(\psi_t(y))$. Then, it holds
			\begin{equation}\label{eq:norm of exp t(V_t(y))}
				\|e^{tV_t(y)}\|\le e^{-t(\lambda_{\min}(D)-\|W\|)}.
			\end{equation}
			\begin{equation}\label{eq:norm-exp-phi}
				\|e^{t\phi_t(y)}\|\le e^{-t(\lambda_{\min}(D)-\|W\|)}.
			\end{equation}
		\end{lemma}
		\begin{proof}
			First, for a fixed $t\ge 0$ and $(x, y)\in(\R^d)^2$, the map $s\in\R_+\mapsto e^{tV_s(y)}x\in\R^d$ is continuous by composition. Now fix $s\ge 0$, let $x\in\R^d$, and define for $t\in\R_+$, $g(t)=|e^{tV_s(y)}x|^2$. Then, $g\in C^1(\R_+)$, $g(0)=|x|^2$ and 
			\begin{IEEEeqnarray}{rCl}\label{eq:dissip proof}
				\frac{\dot{g}(t)}{2}
				&=&-\left\langle De^{tV_s(y)}x, e^{tV_s(y)}x\right\rangle\nonumber\\
				&&+\left\langle WDf(\psi_t(y))e^{tV_s(y)}x, e^{tV_s(y)}x\right\rangle\nonumber\\
				&\le&(-\lambda_{\min}(D)+\|W\|)g(t)
			\end{IEEEeqnarray}
			by Cauchy-Schwarz inequality since $\|Df(\psi_t(y))\|\le 1$. It follows that
			\begin{equation}
				|e^{tV_s(y)}x|^2\le e^{2t(-\lambda_{\min}(D)+\|W\|)}|x|^2.
			\end{equation}
			Taking the $\lim\sup$ as $s\to t$ completes the proof of \eqref{eq:norm of exp t(V_t(y))} by continuity. We prove \eqref{eq:norm-exp-phi} similarly.
		\end{proof}
		
		One also has the following result, where the proof follows the same lines as that of Lemma~\ref{lem:key}.
		\begin{lemma}\label{lem:dissipativity of A}
			Set $A=-D+W$. If $\|W\|\le\lambda_{\min}(D)$, then 
			\begin{equation}\label{eq:dissipativity of A 1}
				\|e^{tA}\|\le 1\qquad\forall t\ge 0.
			\end{equation}
			In particular, if $\|W\|<\lambda_{\min}(D)$, then
			\begin{equation}\label{eq:dissipativity of A 2}
				\|e^{tA}\|< 1\qquad\forall t> 0.
			\end{equation}
		\end{lemma}
	In the following lemma, we show that the series in Proposition~\ref{pro:smooth and Lipschitz activation function} are well-defined.
	\begin{lemma}\label{lem:well-defined xi and chi}
			The series \eqref{eq:sol expansion with remainder}, \eqref{eq:xi}, and \eqref{eq:chi} in Proposition~\ref{pro:smooth and Lipschitz activation function} are well-defined in $\cM_d(\R)$. Moreover, there exists $C=C(t,x^0,D,W,f,B,u)>0$ such that
			\begin{equation}\label{eq:norm zeta}
				\|\zeta_u(t)\|\le C e^{3\Lambda  t}t^3,\qquad\forall t\ge 0.
			\end{equation}
	\end{lemma}
	\begin{proof}
			Let us prove that
			\begin{equation}\label{eq:norm xi}
				\|\xi_u(t)\|\le \Lambda\Lambda _1|Bu|e^{3\Lambda t}t^3
			\end{equation}
			\begin{equation}\label{eq:norm chi}
				\|\chi_u(t)\|\le \Lambda_1\max\left(\Lambda|x^0|+|Bu|,\frac{\|W\|}{\lambda_{\min}(D)}\right)e^{3\Lambda  t}t^3
			\end{equation}
			where $\xi_u(t)$ and $\chi_u(t)$ are defined in \eqref{eq:xi} and \eqref{eq:chi}.
			
			First of all, for every $n\in\N$, one has $1/(n+1)\le 1$ so that
			\begin{equation}
				\left\|\sum_{n=0}^{\infty}\frac{t^{n+1}Z_t^n}{(n+1)!}\right\|\le t\sum_{n=0}^{\infty}\frac{t^{n}\Lambda^n}{n!}=te^{\Lambda t}\qquad\forall t\ge 0
			\end{equation}
			since $\|Z_t^n\|\le\Lambda^n$ by Lemma~\ref{lem:properties on N}. Next, for every $n\ge 1$,
			\begin{equation}\label{eq:for Z}
				\frac{d}{ds}Z_s^n=\sum_{k=1}^{n}Z_s^{k-1}\dot{Z}_sZ_s^{n-k},\quad\left\|\frac{d}{ds}Z_s^n\right\|\le n\Lambda^{n-1}\|\dot{Z}_s\|.
			\end{equation}
			Furthermore, $\dot{Z}_s=WD^2f(\psi_s(x(s)))P_sBu$, and, one finds
			\begin{equation}
				|\xi(t)y|\le \Lambda_1|Bu|t^3e^{3\Lambda t}|y|
			\end{equation}
			by applying \eqref{eq:spectral norm of D_Phi complicated} with $\gamma(s)=-s$ and $\beta(s)=x(s)$. This completes the proof of \eqref{eq:norm xi}. Let us prove \eqref{eq:norm chi}. Firstly, one has $\ddot{x}(s)=DN(x(s))\dot{x}(s)$, $\dot{x}(0)=N(x^0)+Bu$ so that $|\dot{x}(0)|\le\Lambda|x^0|+|Bu|$ since $N(0)=0$ and $N$ in $\Lambda$-Lipschitz by Lemma~\ref{lem:properties on N}. Therefore, by the Cauchy-Schwarz inequality
			\begin{equation*}
				\frac{d}{ds}|\dot{x}(s)|=\frac{\langle \ddot{x}(s),\dot{x}(s)\rangle}{|\dot{x}(s)|}\le-\lambda_{\min}(D)|\dot{x}(s)|+\|W\|
			\end{equation*}
			which, by Gronwall's lemma, implies that
			\begin{equation}\label{eq:nonme x_dot(s)}
				|\dot{x}(s)|\le\max\left(\Lambda|x^0|+|Bu|,\frac{\|W\|}{\lambda_{\min}(D)}\right),\quad\forall s\ge 0.
			\end{equation}
			On the other hand, by applying \eqref{eq:spectral norm of D^2_Phi} with $\gamma(s)=-s$ and $\beta(s)=x(s)$, one gets,
			\begin{equation}\label{eq:spectral norm of second derivative D_Psi^2}
				\|D\psi_s(x(s))\dot{x}(s)\|\le \Lambda _1s e^{2\Lambda s}|\dot{x}(s)|.
			\end{equation}
			One deduces \eqref{eq:norm chi} from \eqref{eq:chi}, \eqref{eq:nonme x_dot(s)} and \eqref{eq:spectral norm of second derivative D_Psi^2}.
	\end{proof}
	
	\subsection{Proof of Proposition~\ref{pro:smooth and Lipschitz activation function}}\label{s:smooth and Lipschitz activation function}
	
	\begin{proof}
		Define for every $n\in\N$, $\eta:t\mapsto\eta(t)=Z_t^nP_t$. Let us show that $\eta$ is absolutely continuous on $[0, T]$ with value in $\cM_d(\R)$. First, one has $\eta(\cdot)\in L^1((0,T); \cM_d(\R))$. In fact,  by applying \eqref{eq:spectral norm of D_Phi complicated} with $\gamma(t)=-t$ and $\beta(t)=x(t)$, one finds
		\begin{equation}
			\int_0^T\|\eta(t)\|\,dt\le\Lambda^n\int_0^Te^{\Lambda t}\,dt=\Lambda^{n-1}(e^{\Lambda T}-1)
		\end{equation}
		where $\Lambda=\lambda_{\max}(D)+\|W\|$. Moreover, $\eta$ is derivable w.r.t. $t$ and $ \dot{\eta}(t)=-Z_tP_t+D^2\psi_t(x(t))\dot{x}(t)$ if $n=0$, while for $n\ge 1$, one has
		\begin{IEEEeqnarray*}{rCl}
			\dot{\eta}(t)&=&\left(\frac{d}{dt}Z_t^n\right)P_t+Z_t^n\dot{P}_t\nonumber\\
			&=&\sum_{k=1}^{n}Z_t^{k-1}\dot{Z}_tZ_t^{n-k}P_t-Z_t^{n}(Z_tP_t-D^2\psi_t(x(t))\dot{x}(t)).
		\end{IEEEeqnarray*}
		By using \eqref{eq:nonme x_dot(s)}, \eqref{eq:spectral norm of second derivative D_Psi^2} and applying \eqref{eq:spectral norm of D_Phi complicated} and \eqref{eq:spectral norm of D^2_Phi} with $\gamma(t)=-t$ and $\beta(t)=x(t)$, one finds for $n=0$,
		\begin{IEEEeqnarray*}{rCl}
			\int_0^T\|\dot{\eta}(t)\|\,dt&\le&
			T\Lambda e^{\Lambda T}\nonumber\\
			&&\hspace{-0.5cm}+T^2\Lambda_1e^{2\Lambda T}\max\left(\Lambda|x^0|+|Bu|,\frac{\|W\|}{\lambda_{\min}(D)}\right).
		\end{IEEEeqnarray*}
		If $n\ge 1$, one uses \eqref{eq:differential second of N} to obtain
		$\dot{Z}_t=WD^2f(\psi_t(x(t)))P_tBu$ and $\|\dot{Z}_t\|\le\Lambda_1e^{\Lambda t}|Bu|$. It follows that
		\begin{IEEEeqnarray*}{rCl}
			\int_0^T\|\dot{\eta}(t)\|\,dt&\le&nT\Lambda^{n-1}\Lambda_1 e^{2\Lambda T}|Bu|+
			T\Lambda^{n+1} e^{\Lambda T}\nonumber\\
			&&\hspace{-1cm}+T^2\Lambda^n\Lambda_1e^{2\Lambda T}\max\left(\Lambda|x^0|+|Bu|,\frac{\|W\|}{\lambda_{\min}(D)}\right).
		\end{IEEEeqnarray*}
		It follows that $\dot{\eta}(\cdot)\in L^1((0,T); \cM_d(\R))$, and that $\eta(\cdot)$ is absolutely continuous on $[0, T]$, which justifies the following integration by parts
		\begin{IEEEeqnarray}{rCl}\label{eq:integration by parts 1}
			\psi_t(x(t)) &=& x^0+\int_{0}^{t}s'P_sBu\,ds=-\int_{0}^{t}\frac{s^2}{2}\left[\frac{d}{ds}Z_s\right]P_sBu\,ds\nonumber\\
			&&\hspace{-2cm}-\int_{0}^{t}sD^2\psi_s(x(s))\dot{x}(s)Bu\,ds -\int_{0}^{t}\frac{s^2}{2}Z_sD^2\psi_s(x(s))\dot{x}(s)Bu\,ds\nonumber\\
			&&\hspace{-1.5cm} +x^0+tP_tBu+\frac{t^2}{2}Z_tP_tBu+\int_{0}^{t}\frac{s^2}{2}Z_s^2P_sBu\,ds
		\end{IEEEeqnarray}
		from which one performs integration by parts on the last integral in \eqref{eq:integration by parts 1} and so on to obtain \eqref{eq:sol expansion with remainder}, \eqref{eq:xi}, and \eqref{eq:chi}, which are well-defined by Lemma~\ref{lem:well-defined xi and chi}.
	\end{proof}

	\subsection{Proof of Proposition~\ref{pro:control synthesis in small time 1}}\label{ss:spectral norm of varphi_u(T)}
	
	\begin{proof} Let us introduce 
		\begin{equation}
			v(t) = \int_0^tD\phi_{T-s}(x(s))Bu\,ds,\qquad t\in[0, T]
		\end{equation}
		where $x(\cdot)$ is the solution of \eqref{eq:Hopfield nonlinear} corresponding to $u\in\R^k$. Applying \eqref{eq:spectral norm of D_Phi complicated} with $\gamma(t)=T-t$ and $\beta(t)=x(t)$, one finds
		\begin{equation}
			|v(t)|\le t|Bu|e^{(\Lambda-\Gamma)t},\qquad t\in[0, T]
		\end{equation}
		showing that $v(t)\underset{t \sim 0}{=}\cO(t)$. Thus, one deduces from \eqref{eq:forward representation} that
		\begin{equation}
			x(t)\underset{t \sim 0}{=}\phi_t(x^0)+\cO(t).
		\end{equation}
		It follows that 
		\begin{IEEEeqnarray*}{rCl}
			Q_{T-t}:= D\phi_{T-t}(x(t))&\underset{t \sim 0}{=}&D\phi_{T-t}(\phi_t(x^0))+\cO(t),\\   
			Z_{T-t}:= DN(\phi_{T-t}(x(t))&\underset{t \sim 0}{=}&A+\cO(t),
		\end{IEEEeqnarray*}
		\begin{equation}
			D^2\phi_{T-t}(x(t))\dot{x}(t)\underset{t \sim 0}{=}D^2\phi_{T-t}(\phi_t(x^0))N(\phi_t(x^0))+\cO(t),
		\end{equation}
		where $A:=DN(\phi_T(x^0))$. Using \eqref{eq:kappa_u} and \eqref{eq:eta_u}, one finds
		\begin{equation}\label{eq:kappa_u and eta approx}
			\kappa_u(T)\underset{T \sim 0}{=}\cO(T^4),\qquad
			\eta_u(T) \underset{T \sim 0}{=} \eta(T)+\cO(T^4),
		\end{equation}
		\begin{equation}
			\eta(T):=\sum_{n=1}^\infty \int_0^T \frac{(t - T)^nA^{n - 1}}{n!} D^2\phi_{T - t}(\phi_t(x^0)) N(\phi_t(x^0))\, dt.
		\end{equation}
		Since $DN(\phi_T(x^0))$ is invertible by assumption, it follows from \eqref{eq:link between the SAs abstract} that
		\begin{IEEEeqnarray}{rCl}\label{eq:eta_u general}
			\eta(T)
			&=&\int_0^T\left[e^{(t-T)A}-\idty\right]D\phi_{T - t}(\phi_t(x^0)) \, dt-\nonumber\\
			&&\hspace{-1.8cm}\int_0^T\left[e^{(t-T)A}-\idty\right]A^{-1}D\phi_{T - t}(\phi_t(x^0)) DN(\phi_t(x^0))\, dt.
		\end{IEEEeqnarray}
		\[
		\hspace{-1cm}\text{Set}\quad\Theta_1:= \int_0^T\left[e^{(t-T)A}-\idty\right]D\phi_{T - t}(\phi_t(x^0)) \, dt.
		\]
		Using integration by parts, one finds
		\begin{IEEEeqnarray}{rCl}\label{eq:Theta 1}
			\Theta_1
			&=&A^{-1}-e^{-TA}A^{-1}D\phi_T(x^0)-\int_0^TD\phi_{T - t}(\phi_t(x^0)) \, dt\nonumber\\
			&&\hspace{-0.3cm}+\int_0^Te^{(t-T)A}A^{-1}D\phi_{T - t}(\phi_t(x^0)) DN(\phi_t(x^0))\, dt.
		\end{IEEEeqnarray}
		since (see, for instance, the proof of Lemma~\ref{lem:on the SC})
		\[
		\frac{d}{dt}D\phi_{T - t}(\phi_t(x^0))=-D\phi_{T - t}(\phi_t(x^0))DN(\phi_t(x^0)).
		\]
		It follows from \eqref{eq:eta_u general} and \eqref{eq:Theta 1} that
		\begin{equation}\label{eq:eta_u reduced}
			\eta(T)=\left(\idty-e^{-TA}\right)A^{-1}D\phi_T(x^0)-\int_0^T\hspace{-0.2cm}D\phi_{T - t}(\phi_t(x^0)) \, dt.
		\end{equation}
		Now, via successive integration by parts, one finds
		\begin{IEEEeqnarray}{rCl}
			\int_0^TD\phi_{T - t}(\phi_t(x^0)) \, dt&=&\int_0^Tt'D\phi_{T - t}(\phi_t(x^0)) \, dt \nonumber\\
			&=&T+\frac{T^2}{2}A+\frac{T^3}{6}A^2\nonumber\\
			&&\hspace{-2.9cm}-\int_0^T\frac{t^2}{2}D\phi_{T - t}(\phi_t(x^0))\left[\frac{d}{dt}DN(\phi_t(x^0))\right]\,dt\nonumber\\
			&&\hspace{-2.9cm}+\int_0^T\left(\frac{t^4}{4!}\right)'D\phi_{T - t}(\phi_t(x^0))DN(\phi_t(x^0))^3 \, dt
		\end{IEEEeqnarray}
		and so on to obtain
		\begin{IEEEeqnarray}{rCl}\label{eq:inter}
			\int_0^TD\phi_{T - t}(\phi_t(x^0)) \, dt&=&\left(e^{TA}-\idty\right)A^{-1}-\nonumber\\
			&&\hspace{-4.2cm}\underbrace{\sum_{n=1}^\infty\int_0^T\frac{t^{n+1}D\phi_{T - t}(\phi_t(x^0))}{(n+1)!}\left[\frac{d}{dt}DN(\phi_t(x^0))^n\right]\,dt}_{\Lambda_T}.
		\end{IEEEeqnarray}
		Letting $Z_t:=DN(\phi_t(x^0))$, one finds
		\begin{equation}\label{eq:for Z inter}
			\frac{d}{dt}Z_t^n=\sum_{k=1}^{n}Z_t^{k-1}\dot{Z}_tZ_t^{n-k},\quad \left\|\frac{d}{dt}Z_t^n\right\|\le n\Lambda^{n-1}\|\dot{Z}_t\|.
		\end{equation}
		Furthermore, $\dot{Z}_t=WD^2f(\phi_t(x^0))N(\phi_t(x^0))$. Therefore,
		\begin{equation}\label{eq:Lambda_T}
			\|\Lambda_T\|\le\Lambda_1 e^{2\Lambda T}|x^0|T^3.
		\end{equation}
		Combining \eqref{eq:Lambda_T}, \eqref{eq:inter} and \eqref{eq:eta_u reduced}, we find
		\begin{IEEEeqnarray}{rCl}\label{eq:eta}
			\eta(T) &\underset{T \sim 0}{=}&  \left(\idty-e^{-TA}\right)A^{-1}D\phi_T(x^0)-\left(e^{TA}-\idty\right)A^{-1}\nonumber\\
			&&+\cO(T^3).
		\end{IEEEeqnarray}
		Combining \eqref{eq:kappa_u and eta approx} and \eqref{eq:eta} yields
		\begin{IEEEeqnarray*}{rCl}
			\varphi_u(T) &\underset{T \sim 0}{=}& \left(\idty-e^{-TA}\right)A^{-1}D\phi_T(x^0)-\left(e^{TA}-\idty\right)A^{-1}\nonumber\\
			&&\hspace{0.3cm}+\cO(T^3),
		\end{IEEEeqnarray*}
		which complete the proof of \eqref{eq:expansion of varphi_u} since $A:=DN(\phi_T(x^0))$. To complete the proof of \eqref{eq:A_T and varphi_u}, one uses \eqref{eq:matrix A_T(x^0)} and \eqref{eq:expansion of varphi_u}.
	\end{proof}
	
	\subsection{Proof of Theorem~\ref{thm:forward implicit nominal-state synthesis step function}}\label{ss:proof of forward implicit nominal-state synthesis step function}
	\begin{proof}
		Let $T\ge\tau$ where $\tau>0$ be such that the assumption in Theorem~\ref{thm:forward implicit nominal-state synthesis step function} is satisfied, and let $u_\tau\in\R^k$. Then, using \eqref{eq:forward representation}, the solution $x(\cdot)$ to~\eqref{eq:Hopfield nonlinear} corresponding to 
		\[
		u(t)=\begin{cases}
			0&\quad\mbox{if}\quad 0\le t\le T-\tau\\
			u_\tau&\quad\mbox{if}\quad T-\tau< t\le T
		\end{cases}
		\]
		reads at time $t=T$ as (see, for instance, Remark~\ref{rmk:on piecewise control})
		\[
		x(T) = \phi_T(x^0)+\int_{T-\tau}^TD\phi_{T-t}(x(t))Bu_\tau\,dt.
		\]
		Arguing as in the proof of Proposition~\ref{pro:smooth and Lipschitz activation function 1}, one expands
		\begin{IEEEeqnarray}{rCl}
			x(T)&=&-\sum_{n=0}^{\infty}\frac{(-\tau)^{n+1}DN(\phi_T(x^0))^{n}}{(n+1)!}D\phi_\tau(\phi_{T-\tau}(x^0))Bu_\tau\nonumber\\
			&&+\phi_T(x^0)-\varphi_u(\tau, T)Bu_\tau
		\end{IEEEeqnarray}
		where $\varphi_{u_\tau}(\tau,T) = \kappa_{u_\tau}(\tau,T) + \eta_{u_\tau}(\tau,T)$, with $\kappa_{u_\tau}(\tau,T)$ and $\eta_{u_\tau}(\tau,T)$ defined in \eqref{eq:kappa_u} and \eqref{eq:eta_u}, respectively, where the integrals are taken over $[T - \tau, T]$.
		
		If $x(T)=x^1$, one argues as in the proof of Theorem~\ref{thm:Forward nominal-state synthesis} to obtain that $u_\tau$ necessarily solves~\eqref{eq:forward implicit nominal-state synthesis step function}, and conversely.
	\end{proof}

	The following results were useful.
	
	\begin{lemma}\label{lem:on the SC}
		Let $x\in\R^d$, and let $T > 0$, $t \in [0, T]$. Then,
		\begin{IEEEeqnarray}{rCl}\label{eq:link between the SAs abstract}
			DN(\phi_T(x))\, D\phi_{T-t}(\phi_t(x)) 
			&=& D\phi_{T - t}(\phi_t(x)) \, DN(\phi_t(x))\nonumber\\
			&&\hspace{-0.8cm}+ D^2\phi_{T - t}(\phi_t(x)) \, N(\phi_t(x)).
		\end{IEEEeqnarray}
		In particular, when $t = 0$, we obtain
		\begin{IEEEeqnarray}{rCl}\label{eq:link between the SAs}
			DN(\phi_T(x)) 
			&=& D\phi_T(x) \, DN(x) \, D\phi_T(x)^{-1}\nonumber\\ 
			&&+ D^2\phi_T(x) \, N(x) \, D\phi_T(x)^{-1}.
		\end{IEEEeqnarray}
	\end{lemma}
	
	\begin{proof}
		Set $y(t)=D\phi_{T-t}(b(t))$ where $b\in C^1(\R_+, \R^d)$. Then, $y$ is derivable by Proposition~\ref{lem:properties on N}, and using the chain rule, one finds
		\begin{equation}\label{eq:first relation}
			\dot{y}(t) = -DN(\phi_{T-t}(b(t)))D\phi_{T-t}(b(t))+D^2\phi_{T-t}(b(t))\dot{b}(t).
		\end{equation}
		On the other and, letting $b(t)=\phi_t(x)$ for some $x\in\R^d$, one finds $y(t)=D\phi_{T-t}(\phi_t(x))$ and thus
		\begin{equation}\label{eq:second relation}
			\dot{y}(t)=-D\phi_{T-t}(\phi_t(x))DN(\phi_t(x)).
		\end{equation}
		In fact, one finds $D\phi_{T-t}(\phi_t(x))D\phi_t(x)=D\phi_T(x)$ from $\phi_{T-t}(\phi_t(x))=\phi_T(x)$, so that, after derivation w.r.t., $t$,
		\[
		\dot{y}(t) D\phi_t(x)=-D\phi_{T-t}(\phi_t(x))DN(\phi_t(x))D\phi_t(x)
		\]
		which yields \eqref{eq:second relation} after right multiplication by $[D\phi_t(x)]^{-1}$. Now, since $\phi_{T-t}(b(t))=\phi_{T-t}(\phi_t(x))=\phi_T(x)$ and $\dot{b}(t)=N(\phi_t(x))$, identifying \eqref{eq:first relation} and \eqref{eq:second relation} yields
		\begin{IEEEeqnarray*}{rCl}
			DN(\phi_T(x))D\phi_{T-t}(\phi_t(x))&=&D\phi_{T-t}(\phi_t(x))DN(\phi_t(x))\nonumber\\
			&&\hspace{-0.8cm}+D^2\phi_{T-t}(\phi_t(x))N(\phi_t(x))
		\end{IEEEeqnarray*}
		completing the proof of \eqref{eq:link between the SAs abstract}. Finally, \eqref{eq:link between the SAs} follows by evaluating \eqref{eq:link between the SAs abstract} at $t=0$, and right multiplication by $[D\phi_T(x)]^{-1}$.
	\end{proof}

\end{document}